\documentclass{elsarticle}
\usepackage{amsmath, amsthm, amssymb, amsfonts, mathtools, bbold, stmaryrd}
\usepackage{tikz-cd}
\usepackage[utf8]{inputenc}
\usepackage[T1]{fontenc}
\usepackage{mdframed}
\usepackage{xcolor}
\usepackage{hyperref}

\theoremstyle{plain}
\newtheorem{thm}{Theorem}[section]
\newtheorem{lemma}[thm]{Lemma}
\newtheorem{prop}[thm]{Proposition}
\newtheorem{cor}[thm]{Corollary}

\theoremstyle{definition}
\newtheorem{defn}[thm]{Definition}

\newtheorem{exmp}[thm]{Example}
\newtheorem{rmk}[thm]{Remark}
\newtheorem{alg}{Algorithm}

\newcommand{\m}{\mathfrak{m}}

\renewcommand{\k}{\mathbb{K}}

\renewcommand{\a}{\alpha}
\renewcommand{\b}{\beta}
\renewcommand{\d}{\delta}
\newcommand{\z}{\zeta}
\newcommand{\de}{\partial}
\newcommand{\se}{\subseteq}

\renewcommand{\bf}{\mathbf}
\newcommand{\x}{\bf{x}}
\newcommand{\tn}{\textnormal}

\newcommand{\Ann}{\tn{Ann}}

\newcommand{\rk}{\tn{rk}}
\newcommand{\Hom}{\tn{Hom}}

\newcommand{\kdim}{\dim_\tn{krull}}
\newcommand{\Z}{\mathbb{Z}}
\newcommand{\N}{\mathbb{N}}
\newcommand{\Cx}{\mathbb{C}}

\renewcommand{\H}{\mathbb{H}}
\newcommand{\rH}{\H^{\square}}
\newcommand{\1}{\mathbb{1}}
\newcommand{\0}{\mathbf{0}}

\newcommand{\VV}{\mathcal{V}}

\newcommand{\A}{\mathcal{A}}
\newcommand{\B}{\mathcal{B}}
\newcommand{\C}{\mathcal{C}}
\newcommand{\M}{\mathbb{M}}

\begin{document}

\begin{abstract}(EN)
We revise the famous algorithm for symmetric tensor decomposition due to Brachat, Comon, Mourrain and Tsidgaridas.
Afterwards, we generalize it in order to detect possibly different decompositions involving points on the tangential variety of a Veronese variety.
Finally, we produce an algorithm for cactus rank and decomposition, which also detects the support of the minimal apolar scheme and its length at each component. 
\\
\medskip
\\
(FR) Nous revenons sur le fameux algorithme de Brachat, Comon, Mourrain et 
Tsidgaridas pour la d\'composition des tenseurs sym\'etriques. Ensuite, 
nous le g\'en\'eralisons afin de d\'etecter de possibles d\'ecompositions 
diff\'erentes impliquant des points sur la vari\'et\'e tangentielle d'une 
vari\'et\'e de Veronese. Enfin, nous proposons un algorithme pour le rang et 
la d\'ecomposition cactus, qui, lui aussi, d\'etecte le support du sch\'ema 
apolaire minimal ainsi que sa longueur sur chaque composante.\end{abstract}


\begin{frontmatter}
\title{Waring, tangential and cactus decompositions}
\author{Alessandra Bernardi}
\ead{alessandra.bernardi@unitn.it}
\author{Daniele Taufer}
\ead{daniele.taufer@gmail.com}
\date{\today}
\end{frontmatter}
\tableofcontents

\section{Introduction}
Symmetric Tensor Decomposition (SymTD) is one of the most active research topic of the last decades and it has received many attentions both from the pure mathematical community and applied ones (signal processing \cite{comon}, phylogenetics \cite{ar}, quantum information \cite{glw,hjn,bercar}, computational complexity \cite{land}, geometric modeling of shapes \cite{ev}). The push towards the generation of algorithms that efficiently compute a specific type of decomposition of given symmetric tensors has not only a practical interest but also extremely deep theoretical facets.

The problem can be rephrased as follows: given a homogeneous polynomial $F$ of degree $d$ (i.e. an order-$d$ symmetric tensor), find the minimum number of linear forms $L_1, \ldots, L_r$ such that
$$F=\sum_{i=1}^r L_i^d.$$

Such a decomposition is known as \emph{symmetric tensor decomposition}, \emph{Waring decomposition} (this is the one we use all along the paper), \emph{symmetric rank decomposition}, \emph{minimal symmetric CP decomposition} or \emph{symmetric canonical polyadic decomposition}. The minimum integer $r$ realizing this decomposition is called the \emph{Waring rank} of $F$.

Despite this problem has many equivalent formulations in the tensor community, we state and study it only in terms of homogeneous polynomials in place of symmetric tensors since they may be easily identified.
This allows us to define and prove precisely the tools on which the proposed algorithms rely, by working within a purely algebraic frame.

\medskip

For binary forms, the solution to this problem is well-known from the late XIX century thanks to J.J. Sylvester \cite{Sy} and more recently revised in \cite{cs,bgi,SymTensorDec}. The first significant improvement for any number of variables was due to A. Iarrobino and V. Kanev \cite[Section 5.4]{IK} who extended Sylvester's idea to any number of variables (their approach is sometimes referred to as \emph{catalecticant method}). Their idea works only if the Waring rank of the given polynomial is equal to the rank of a certain matrix (which we call the ``largest numerical Hankel matrix'', see Section \ref{Sec:minr}).

In 2013 L. Oeding and G. Ottaviani used vector bundles techniques and representation theory to give an algorithm \cite[Algorithm 4]{oo} for Waring rank that, as the Iarrobino--Kanev idea, works only if the Waring rank of the polynomial is smaller than the rank of a matrix constructed with their techniques.

Nowadays, one of the best ideas to generalize those methods is due to J. Brachat, P. Comon, B. Mourrain and E. Tsidgaridas that in \cite{SymTensorDec} developed an algorithm that gets rid of the restrictions imposed by the usage of catalecticant matrices. Their idea is to employ the so-called \emph{Hankel matrix}, which in a way encodes the information of every catalecticant matrix. 
This idea has been used to generalize such an algorithm to other structured tensors (cf. \cite{BBCM, BBCM0,  IAB, ABMM}). 
A detailed presentation of this subject may be found in \cite{bccgo}.

It is worth noting that all the quoted algorithms are symbolic; nevertheless also a numerical algorithm \cite{bdhm} based on homotopy continuation running in Bertini system \cite{Bertini} has been developed.

\medskip

The first part of our paper consists of a revision of the Brachat--Comon--Mourrain--Tsidgaridas' algorithm: we propose various improvements, both from the theoretical point of view and of computational efficiency.

\medskip

The second part of the paper is devoted to a different kind of decomposition, which we call \emph{tangential decomposition}. We look for the minimal way of writing a given homogeneous polynomial $F$ of degree $d$ as 
\begin{equation}\label{tangential:decomp:intro}
    F=\sum_{i=1}^s L_i^{d-1}M_i
\end{equation}
where $L_i$'s and $M_i$'s are not necessarily distinct linear forms and the minimality is on the number of possibly repeated linear forms appearing in the decomposition.
We call it tangential decomposition since the projective classes of the addenda appearing in such a decomposition are points on the tangential variety of a Veronese variety \cite{cgg,bcgi1,bcgi2,av,bb3}.
We generalize the SymTD algorithm to explicitly detect tangential decompositions: cf. Section \ref{Sec:TangentialAlgorithm}.

\medskip

The last part of the present manuscript is devoted to a slightly different but related concept: the \emph{cactus rank} of a homogeneous polynomial. It was firstly introduced in \cite{Bernardi:2011vn,MR2842085} following the ideas of \cite{Buczynska:2010kx} but it was already present in the literature as \emph{scheme length}, cf. \cite{IK}.
The cactus rank has been defined as the minimal length of an apolar zero-dimensional scheme (cf. Section \ref{apolarity:section} for a formal definition of apolarity). In \cite[Theorem 3.7]{bbm} the cactus rank of a homogeneous polynomial $F$ was proved to coincide with the size of a generalized decomposition (cf. \eqref{cactus:decomp} in Theorem \ref{Thm:Structure}) of a certain map associated to $F$ and this is the definition of cactus rank that we use here (cf. Definition \ref{cactus:definition}). From a polynomial decomposition point of view, finding a \emph{cactus decomposition} of a given homogeneous polynomial $F$ amounts to writing it, in a certain minimal way, as
\begin{equation}\label{cactus:decomp:on:F}
    F=\sum_{i=1}^s L_i^{d-{k_i}}N_i,
\end{equation}
where the $L_i$'s are linear forms and the $N_i$'s are homogeneous polynomials of degree $k_i$ for certain $k_i < d$.
The minimality of the above decomposition is on the sum of the dimensions of the spaces of derivatives killing the $N_i$'s and such a minimum is the cactus rank of $F$.

The importance of the cactus rank of a polynomial is witnessed in various purview. First of all it is an appealing topic because of its interpretation as the length of certain Gorenstein zero-dimensional schemes \cite{IK,Buczynska:2010kx,Bernardi:2011vn,bjmr}.
Secondly, many results and algorithms for the Waring rank computation have been discovered by studying the relation between the Waring rank and the cactus rank of a given polynomial \cite{bgi,bb1,bb2,bj}. Moreover, it is connected to the study of joins of osculating varieties of Veroneseans \cite{bcgi1,bcgi2,ccgo,av,flos}. Last but not least, J.M. Landsberg in his recent book \cite[p.299]{land} states that the knowledge of the cactus rank of a generic cubic form (computed in \cite{Bernardi:2011vn}) implies the impossibility of proving superlinear border rank bounds for tensors with determinantal equations.

We conclude our paper by making use of the developed algebraic tools to produce a symbolic algorithm for computing the cactus rank of any homogeneous polynomial along with many information about the generalized decomposition, such as the linear forms $L_i$'s appearing in \eqref{cactus:decomp:on:F}, 
 a bound on their exponents $d-k_i$ and the dimension of the parameter space where each $N_i$ can be minimally looked for.
A reader who is familiar with the original definition of cactus rank may note that our last algorithm computes the support of the minimal apolar scheme together with the length of all the subschemes supported at a single point.

\subsection{Novel contribution}

The novel results of our work are numerous.
\\
An algorithm (Algorithm \ref{alg:waring}) for Waring decomposition is presented, it revises the one in \cite{SymTensorDec} in a multitude of aspects:
\begin{itemize}
\item using essential variables we avoid incorrect outputs (Section \ref{section:essential:variables}); 
\item starting with the maximal rank of the Hankel matrix we do not miss any good decomposition and reduce the number of iterations (Section \ref{Startingr});
\item restricting the criteria on bases to be tested, namely those which are complete staircases, we improve the algorithm performance. A rigorous proof that a minimal decomposition always arises from these bases is also provided (Section \ref{requirementB});
\item testing eigenvectors instead of looking at eigenvalues we always avoid outputs that are not minimal (Section \ref{section:Looking:eigenvectors}). Despite this procedure may seem unnecessarily difficult for \emph{generic} polynomials, we prove that it is theoretically necessary to ensure minimality.
\end{itemize}
Algorithm \ref{alg:tg} for tangential decomposition and Algorithm \ref{alg:cactus} for cactus decomposition are new, as well as the theorems on which they rely.

\subsection{Structure of the paper}
The paper is organized as follows: Section \ref{prelim} contains the algebraic tools needed by the algorithms. Section \ref{algo:section} is the core of our version of the SymTD algorithm. In Section \ref{algo:advantages} we analyze the advantages of our formulation of this algorithm. 
Section \ref{Section:Tangential} is devoted to the specific case of tangential decomposition: we give an algorithm for explicitly computing the minimal weighted $s$ for which the decomposition \eqref{tangential:decomp:intro} is possible, as well as recovering all the linear forms involved.
We conclude the paper with Section \ref{section:cactus:rank}, where we give an algorithm for computing the cactus rank, the linear forms involved in \eqref{cactus:decomp:on:F}, the length of the apolar scheme at each point of its support and the cactus decomposition of $F$.
We provide many non-trivial examples of the proposed algorithms, highlighting their crucial steps.

\section{Preliminaries}\label{prelim}

\subsection{Notation}
In this paper, $\k$ is an algebraically closed field of zero-characteristic, $n$ is a positive integer, $R = \k[\x] = \k[x_1, \dots, x_n]$ is the ring of polynomials in $n$ variables over $\k$ and $R^*$ its dual.
For any non-negative integer $d$ we also denote by $R_{\leq d}$ the $\k$-space of polynomials of total degree at most $d$ and with $R^h_d$ the homogeneous degree $d$ polynomials in $n+1$ variables $x_0, x_1, \dots, x_n$. For every $F \in R^h_d$ we denote by the lowercase letter $f \in R_{\leq d}$ the corresponding dehomogenization with respect to $x_0$, namely $f(x_1, \dots, x_n) = F(1, x_1, \dots, x_n)$.

Given a subset $J \se R$, we denote the affine variety defined by $J$ as $\VV(J) = \{P \in \k^n \ | \ \forall \ f \in J,\ f(P) = 0\}$.

Given a point $\z \in \k^n$, we denote by $\1_{\z} \in R^*$ the evaluation-in-$\z$ morphism. Moreover, for every operator $\Delta \in \Hom_{\k}(R)$ we use the subscript $\z$ to denote the post-composition with $\1_{\z}$, namely $\Delta_{\z} = \1_{\z} \circ \Delta \in R^*$.

Furthermore, given a polynomial $f \in R$, we denote the corresponding differential polynomial $f(\d) \in \Hom_{\k}(R)$, obtained by substituting the $i$-th variable in $f$ with the $i$-th partial derivation and by interpreting the product of derivations as the composition.

Finally, we make use of the standard multi-index notation: for every vector $\a = (\a_1, \dots, \a_n) \in \N^n$ we denote by $|\a| = \sum_{i = 1}^n \a_i$, $\a ! = \prod_{i = 1}^n (\a_i!)$ and when $d \geq |\a|$ we set ${d \choose \a} = \frac{d!}{\a! (d - |\a|)!}$.
We also use a short form $\d^{\a}$ to denote the differential monomial $\x^{\a}(\d) = \de_1^{\a_1} \circ \dots \circ \de_n^{\a_n}$.

\begin{defn} Let $F \in R^h_d$. We define the \emph{Waring rank} of $F$ as the minimal $r \in \N$ such that there exist $\{\lambda_1, \dots, \lambda_r\} \se \k$ and $\{L_1, \dots, L_r\} \se R^h_1$ with
\begin{equation*}
F = \sum_{i = 1}^r \lambda_i L_i^d.
\end{equation*}
Such a decomposition for which $r$ is minimal is called \emph{Waring decomposition}.
Similarly, a \emph{Waring decomposition} of $f \in R_{\leq d}$ is
\begin{equation*}
f = \sum_{i = 1}^r \lambda_i l_i^d.
\end{equation*}
with $\{\lambda_1, \dots, \lambda_r\} \se \k$, $\{l_1, \dots, l_r\} \se R_{\leq 1}$ and $r$ minimal.
\end{defn}

The Waring rank is well-defined, i.e. for every homogeneous polynomial a Waring  decomposition exists.
In fact, the $d$-th Veronese variety, which parameterizes projective classes of $d$-th powers of linear forms, is a complex non degenerate projective variety in $\mathbb{P}^{ {n+d \choose d} }(\k)$.
However, it is also known that this Waring  decomposition might well be not unique (cf. \cite{Sy,Hi,Ri,Pa,Me,gm,Ci,RV}).

\subsection{Algebraic tools}

We need the classical characterization of zero-dimensional ideals, which we summarize in the following theorem.

\begin{thm}[Zero-dimensional ideals] \label{thm:EquivCharactZeroDim} Let $I$ be a proper ideal of $R$. The following are equivalent.
\begin{enumerate}[(i)]
\item $I$ is zero-dimensional, i.e. $\kdim R/I = 0$.
\item $\dim_\k (R/I) < \infty$.
\item $\#\VV(I) < \infty$.
\end{enumerate}
Moreover, if $I$ is zero-dimensional, we have
\begin{equation*}
 \#\VV(I) \leq \dim_\k (R/I),
\end{equation*}
which is an equality if and only if $I$ is also radical.
\end{thm}

\proof See \cite{UsingGeometry}, \cite{Atiyah} or Theorem 6, Proposition 8 in Section 5.3 of \cite{IdealVarieties}.
\endproof

Given an $R$-algebra $\A$, its dual $\A^*$ has a natural $\A$-module structure given for every $a \in \A$ and $\Lambda \in \A^*$ by $a \star \Lambda = \big(b \mapsto \Lambda(ab)\big) \in \A^*$.
Thus, we define the multiplication operators by $a \in \A$ as
\begin{align*}
 M_a : \A &\to \A, & M_a^t : \A^* &\to \A^*,\\
 b &\mapsto ab, & \varphi &\mapsto a \star \varphi.
\end{align*}

\begin{defn}\label{def:Hankel} Let $\Lambda \in R^*$. We define its \emph{Hankel operator} to be the $R$-module morphism
 \begin{align*}
  H_{\Lambda} : R &\to R^*,\\
  f &\mapsto f \star \Lambda.
 \end{align*}
 Moreover, we denote its kernel by $I_{\Lambda} = \ker H_{\Lambda}$.
\end{defn}

For every $\Lambda \in R^*$ we see that $I_{\Lambda}$ is an ideal of $R$ and by defining $\A_{\Lambda} = R/I_{\Lambda}$ we have $\rk H_{\Lambda} = \dim_{\k} \A_{\Lambda}$.

Let now $I \se R$ be a zero-dimensional ideal, so that $\VV(I) = \{\z_1, \dots, \z_d\}$.
Since for every $i \in \{1, \dots, d\}$ we have $I \se \ker \1_{\z_i}$, then we may consider the restrictions of $\1_{\z_i}$ to $\A = R/I$, that we denote in the same way with a slight abuse of notation.

\begin{thm}\label{Eigenvectors} Let $I\subseteq R$ be a zero-dimensional ideal, with $\VV(I) = \{\z_1, \dots, \z_d\}$. For every $a \in \A = R/I$ the following hold.
\begin{enumerate}[(i)]
    \item The eigenvalues of $M_{a}$ and $M_{a}^t$ are $\{a(\z_1), \dots, a(\z_d)\}$.
    \item An element $v \in \A^*$ is an eigenvector for every $\{M_{x_i}^t\}_{i \in \{1, \dots, n\}}$ if and only if there are $j \in \{1, \dots, d\}$ and $k \in \k\setminus\{0\}$ such that $v = k \1_{\z_j}$.
\end{enumerate}
\end{thm}
\proof Part $(i)$ is Theorem (4.5) in Chapter 2, \S 4 of \cite{UsingGeometry}. Both parts are proved in \cite[Thm. 4.23]{FrenchRef}.
%
\endproof

\begin{cor}\label{RadicalDistinctEigenvectors} Let $\Lambda \in R^*$ such that $A_{\Lambda}$ is an $r$-dimensional $\k$-vector space. Then the following are equivalent.
\begin{enumerate}[(i)]
    \item Up to $\k$-multiplication, there are exactly $r$ distinct common eigenvectors of $\{M_{x_i}^t\}_{i \in \{1, \dots, n\}}$.
    \item $I_{\Lambda}$ is radical.
\end{enumerate}
\end{cor}
\proof By Theorem \ref{thm:EquivCharactZeroDim} since $\dim_{\k} A_{\Lambda} = r$ then $I_{\Lambda}$ is zero-dimensional so $\VV(I_{\Lambda}) = \{\z_1, \dots, \z_d\}$ and it is radical if and only if $r = d$.
By Theorem \ref{Eigenvectors} we have that, up to scalar multiplication, the distinct eigenvectors common to $\{M_{x_i}^t\}_{i \in \{1, \dots, n\}}$ are $\{\1_{\z_1}, \dots, \1_{\z_d}\}$, thus $I_{\Lambda}$ is radical if and only if these common eigenvectors are exactly $r$.
\endproof

\subsection{Apolarity} \label{apolarity:section}

\begin{defn}\label{apolar:ideal} Given a set $S \se R$, its \emph{apolar set} is
\begin{equation*}
 S^{\perp} = \{\Lambda \in R^* \ | \ \forall f \in S, \ \Lambda(f) = 0\}.
\end{equation*}
If $I \se R$ is an ideal, $I^{\perp}$ is referred to as its \emph{apolar ideal}.
\end{defn}

For every $\z = (\z_1, \dots, \z_n) \in \k^n$ let $\m_{\z} = (x_1 - \z_1, \dots, x_n - \z_n) \se R$ be the corresponding maximal ideal. The apolar ideal of any zero-dimensional ideal is completely determined in terms of the apolar ideals of its primary components.

\begin{thm} \label{Thm:Structure} Let $I \subseteq R$ be a zero-dimensional ideal, $\VV(I) = \{\z_1, \dots, \z_d\}$. Then the minimal primary decomposition of $I$ is given by $I = Q_1 \cap \dots \cap Q_d$ where $Q_i$ is $\m_{\z_i}$-primary and
\begin{equation*}
 I^{\perp} = Q_1^{\perp} \oplus \dots \oplus Q_d^{\perp}.
\end{equation*}
Furthermore, for every $\Lambda \in I^{\perp}$ there are $\{p_i\}_{i \in \{1, \dots, d\}} \subseteq R$ such that
\begin{equation}\label{cactus:decomp}
 \Lambda = \sum_{i = 1}^d \1_{\z_i} \circ p_i(\d).
\end{equation}
Moreover, if $I$ is also radical then $\{p_i\}_{i \in \{1, \dots, d\}} \subseteq \k$ are all constants.
\end{thm}

\proof See \cite[Theorems 7.34 and 7.5]{FrenchRef}.
\endproof

In the setting of our algorithms the considered ideal $I$ will be $I_{\Lambda}$. By definition $\Lambda \in I_{\Lambda}^{\perp}$, hence by Theorem \ref{Thm:Structure} we have decompositions of $\Lambda$ as in \eqref{cactus:decomp}.
The proof of Theorem \ref{Thm:Structure} given in \cite{FrenchRef} also shows that
\begin{equation*}
    \dim_{\k} Q_i^{\perp} = \dim_{\k} \langle \{ \1_{\z_i} \circ \de^{\alpha}p_i \}_{|\alpha|\leq \deg p_i}\rangle_{\k}.
\end{equation*}
The above quantity is usually called the \emph{multiplicity of} $\z_i$.
A decomposition as in \eqref{cactus:decomp} for which the sum of the multiplicities $r = \sum_{i=1}^d \dim_{\k} Q_i^{\perp} = \dim_{\k} I_{\Lambda}^{\perp}$ is minimal is called \emph{generalized decomposition of $\Lambda$} (cf. \cite{SymTensorDec,bbm}) and such an $r$ is referred to as the \emph{size} of the generalized decomposition.

We want to exploit the knowledge of inverse systems to address the Waring decomposition problem, by formulating an equivalent decomposition problem in the dual space. For this reason we associate a linear form to every polynomial by defining the apolar product over $R_{\leq d}$ as
\begin{equation*}
\left\langle \sum_{|\a| \leq d} f_{\a}\x^{\a}, \sum_{|\a| \leq d} g_{\a}\x^{\a} \right\rangle = \sum_{|\a| \leq d} \frac{f_{\a}g_{\a}}{ {d \choose \a} }.
\end{equation*}

\begin{defn} Let $f \in R_{\leq d}$. We define its \emph{dual polynomial} $f^* \in R_{\leq d}^*$ as
\begin{align*}
f^* : R_{\leq d} &\to \k,\\
g &\mapsto \langle f, g \rangle .
\end{align*}
We also define the \emph{dual map} as
\begin{align*}
\tau : R_{\leq d} &\to R_{\leq d}^*,\\
f &\mapsto f^* .
\end{align*}
\end{defn}
It is easy to see that the apolar product is a $\k$-bilinear, symmetric and non-degenerate form on $R_{\leq d}$, hence $\tau$ is an injective morphism of $\k$-modules.

\begin{prop}\label{Prop:DualLinearAffine} Let $l = 1 + l_1x_1 + \dots + l_nx_n \in R_{\leq 1}$. For every $f \in R_{\leq d}$ we have
\begin{equation*}
\tau( l^d )(f) = \1_{(l_1, \dots, l_n)}(f).
\end{equation*}
\end{prop}
\proof By a straightforward computation we get
\begin{equation*}
l^d = \sum_{|\a| \leq d} {d \choose \a} (l_1, \dots, l_n)^{\a}\x^{\a}.
\end{equation*}
Thus, for every $f = \sum_{|\a| \leq d} f_{\a}\x^{\a} \in R_{\leq d}$ we have
\begin{equation*}
(l^d)^* (f) = \left\langle \sum_{|\a| \leq d} {d \choose \a} (l_1, \dots, l_n)^{\a}\x^{\a}, \sum_{|\a| \leq d} f_{\a}\x^{\a} \right\rangle = \sum_{|\a| \leq d} f_{\a}(l_1, \dots, l_n)^{\a},
\end{equation*}
which is exactly the evaluation of $f$ in $(l_1, \dots, l_n)$.
\endproof

We abbreviate the above notation by writing $\1_{l}(f)$ in place of $\1_{(l_1, \dots, l_n)}(f)$.

\section{Waring decomposition algorithm}\label{algo:section}

\subsection{Problem reformulation}

By a generic change of coordinates, finding a Waring decomposition of a given $F \in R^h_d$ is equivalent to finding a Waring decomposition of the corresponding $f \in R_{\leq d}$.

Since the dual map is $\k$-linear and injective, the problem of finding a Waring decomposition of $f \in R_{\leq d}$ is equivalent to finding the minimal $r \in \N$ and $\{\lambda_1, \dots, \lambda_r\} \se \k$, $\{l_1, \dots, l_r\} \se R_{\leq 1}$ such that
\begin{equation*}
f^* = \tau(f) = \tau\left(\sum_{i = 1}^r \lambda_i l_i^d\right) = \sum_{i = 1}^r \lambda_i \1_{l_i},
\end{equation*}
where the last equality follows from Proposition \ref{Prop:DualLinearAffine}.

Instead of searching for $f^* \in R_{\leq d}^*$, we look for a $\Lambda = \sum_{i = 1}^r \lambda_i \1_{l_i} \in R^*$ which \emph{extends}\label{extension} $f^*\in R^*_{\leq d}$, meaning that $\Lambda(x^{\alpha})=f^*(x^\alpha)$ for every $\alpha\in \mathbb{N}^n$ with $|\alpha |\leq d$, or equivalently the degree $\leq d$ part of $\Lambda$ coincides with $f^*$.

Once such $\Lambda$ is found, by restricting the evaluation maps we forthwith find  $\Lambda|_{R_{\leq d}} = f^* = \sum_{i = 1}^r \lambda_i \1_{l_i} \in R_{\leq d}^*$, which leads to a Waring decomposition of $f$.

The problem of finding such an extension $\Lambda \in R^*$ is equivalent to finding $\Lambda \in R^*$ which has prescribed properties on its Hankel operator.

\begin{thm}\label{PbReformulation} Let $\Lambda \in R^*$. The following are equivalent.
\begin{enumerate}[(i)]
    \item There exist non-zero constants $\{\lambda_1, \dots, \lambda_r\} \se \k$ and distinct points $\{\z_1, \dots, \z_r\} \se \k^n$ such that
\begin{equation*}
\Lambda = \sum_{i = 1}^r \lambda_i \1_{\z_i},
\end{equation*}
    \item $\rk H_{\Lambda} = r$ and $I_{\Lambda}$ is radical.
\end{enumerate}
\end{thm}
\proof See Theorem 5.9 in \cite{SymTensorDec}.
\endproof

Thus, our intention is to come up with $\Lambda \in R^*$ which extends $f^*$ and has the minimal $r = \rk H_{\Lambda}$ for which $I_{\Lambda}$ is radical.
By considering the monomial basis $\{\x^{\a}\}_{\a \in \N^n}$ on $R$ and its dual basis on $R^*$, namely $\{\frac{1}{\a!}\d_\0^{\a}\}_{\a \in \N^n}$, the matrix of the Hankel operator $H_{\Lambda}$ is $\H_{\Lambda} = \big(\Lambda(\x^{\a+\b})\big)_{\a, \b \in \N^n}$.
Since we want it to agree with $f^*$ on $R_{\leq d}$, we consider the \emph{generalized Hankel matrix} $\H_{\Lambda}(\mathbf{h})$ defined by using variables $\{\mathbf{h}_{\a}\}_{\substack{\a \in N^n\\ \hspace{-0.15cm}|\a| > d}}$ where $f^*$ is not defined:
\begin{equation*}
\big(\H_{\Lambda}(\mathbf{h})\big)_{\a, \b \in \N^n} = \begin{cases}
f^*(\x^{\a+\b}) & \text{ if } |\a + \b| \leq d,\\
\mathbf{h}_{\a+\b} & \text{ if } |\a + \b| > d.
\end{cases}
\end{equation*}

Given a finite monomial set $B \se R$ we denote by $H_{\Lambda}^B$ the restriction of $H_{\Lambda}$  $H_{\Lambda}^B: \langle B \rangle_{\k} \to \langle B \rangle_{\k}^*$ and by $\H_{\Lambda}^B$ the matrix of this map with respect to $B$ and its dual basis.
In this setting a direct check shows that if $B = \{b_1, \dots, b_r\}$ then $\H_{\Lambda}^B = \big(\Lambda(b_i b_j)\big)_{1 \leq i, j \leq r}$.

\begin{prop}\label{Bbasis} Let $\Lambda \in R^*$ and $B = \{b_1, \dots, b_r\} \se \A_{\Lambda}$. Then $\rk H_{\Lambda} = r$ and $\H_{\Lambda}^B$ is invertible if and only if $B$ is a $\k$-basis of $\A_{\Lambda}$.
\end{prop}
\proof Let $\H_{\Lambda}^B$ be invertible and $\dim_{\k} \A_{\Lambda} = \rk H_{\Lambda} = r$. It is sufficient to show that $b_1, \dots, b_r$ are linearly independent. By contradiction, assume that for some not all zeros constants $k_i \in \k$ we have $\sum_{i = 1}^r k_i b_i \in I_{\Lambda}$. Then the same non-trivial combination between the columns of $\H_{\Lambda}^B$ gives
\begin{equation*}
\sum_{i = 1}^r k_i\left(\begin{array}{c}
\Lambda(b_1b_i)\\
\vdots\\
\Lambda(b_rb_i)
\end{array}\right) = \left(\begin{array}{c}
\Lambda(b_1 \sum_{i = 1}^r k_ib_i )\\
\vdots\\
\Lambda(b_r \sum_{i = 1}^r k_ib_i )
\end{array}\right) = \mathbf{0},
\end{equation*}
contradicting the invertibility of $\H_{\Lambda}^B$.

Let now $B$ be a $\k$-basis of $\A_{\Lambda}$.
Again $r = \dim_{\k} \A_{\Lambda} = \rk H_{\Lambda}$ so we just need to show that the columns of $\H_{\Lambda}^B$ are linearly independent.
Assume by contradiction that we have a non-trivial vanishing combination of the columns with coefficients $\{k_1, \dots, k_r\} \se \k$.
As above, this implies that
\begin{equation*}
\forall j \in \{1, \dots, r\} \ : \ \Lambda(b_j \sum_{i = 1}^r k_ib_i ) = 0.
\end{equation*}
Since $B$ is a $\k$-basis of $\A_{\Lambda}$ then for every $f \in R$ there are $u_i \in \k$ and $\iota \in I_{\Lambda}$ such that $f = u_1 b_1 + \dots + u_r b_r + \iota$. Therefore
\begin{align*}
 \Big( \sum_{i = 1}^r k_ib_i \Big) \star \Lambda (f) &= \Lambda\Big( f \sum_{i = 1}^r k_ib_i \Big) = \Lambda\Big((u_1 b_1 + \dots + u_r b_r)\sum_{i = 1}^r k_ib_i\Big) \\
 &= u_1 \Lambda\Big(b_1\sum_{i = 1}^r k_ib_i\Big) + \dots + u_r \Lambda\Big(b_r\sum_{i = 1}^r k_ib_i\Big) = 0.
\end{align*}
Hence, we conclude that $\sum_{i = 1}^r k_i b_i \in I_{\Lambda}$, contradiction.
\endproof

\begin{lemma}\label{MultiplicationHankel} Let $\Lambda \in R^*$ such that $\rk H_{\Lambda} < \infty$ and $B$ be a set of generators of $\A_{\Lambda}$ as $\k$-vector space.
Then for every $a \in A_{\Lambda}$ we have
\begin{equation*}
H_{a \star \Lambda}^B = M_a^t \circ H_{\Lambda}^B \in \Hom_R(\A_{\Lambda}, \A_{\Lambda}^*).
\end{equation*}
\end{lemma}
\proof For every $f \in \A_{\Lambda}$ we have
\begin{equation*}
H_{a \star \Lambda}^B(f) = f \star a \star \Lambda = a \star f \star \Lambda = M_a^t \circ H_{\Lambda}^B(f).
\end{equation*}
Thus, $H_{a \star \Lambda}^B = M_a^t \circ H_{\Lambda}^B$ as morphisms of $R$-modules.
\endproof

By Proposition \ref{Bbasis} and Lemma \ref{MultiplicationHankel} if $B$ is a $\k$-basis of $\A_{\Lambda}$ we can construct the matrices $(\M_{x_i}^B)^t = \H_{x_i \star \Lambda}^B (\H_{\Lambda}^B)^{-1}$ of the multiplication-by-$x_i$ operators on $\A_{\Lambda}^*$. Since these are the matrices of $M_{x_i}^t$, they have to commute.

By Theorem \ref{Eigenvectors} the common eigenvectors of $(\M_{x_i}^B)^t$ are $\{\1_{\z}\}_{\z \in \VV(I_{\Lambda})}$.
If $I_{\Lambda}$ is radical (equivalently, by Corollary \ref{RadicalDistinctEigenvectors}, if $|B| = |\VV(I_{\Lambda})|$) then by Theorem \ref{PbReformulation} we have found a Waring decomposition of $f$. In fact, once we have the eigenvector corresponding to $\1_{\z}$ we read the coefficients of the affine linear form $1 + \z_1x_1 + \dots + \z_nx_n$ appearing in the decomposition of $f$ by evaluating $\1_{\z}(x_i)$.
Since in $\A_{\Lambda}^*$ we have been using the dual basis of $B$, this means that these coefficients may be read directly from the $x_i^*$-entry of the eigenvectors, when $x_i \in B$.

Thus, the problem of finding a Waring decomposition of $f$ may be solved by finding constants to plug in $\H_{\Lambda}(\mathbf{h})$ in order to have a basis $B$ satisfying all the previous conditions, with $|B|$ minimal.

\subsection{Choice of the basis}

In this section we show that among the possible bases $B$ there are special choices that we might consider in order to reduce the number of tests performed by the algorithm. 
First, we observe that our bases may always be composed of elements of degree bounded by $\deg F$, where $F$ is the polynomial under consideration.

\begin{prop} \label{Degd}
Let $F \in R^h_d$ and $\Lambda \in R^*$ be an extension of $f^* \in R_{\leq d}^*$.
Then there is a monomial basis $B$ of $\A_{\Lambda}$ such that every $[b] \in B$ admits a representative with $\deg b \leq d$.
\end{prop}
\begin{proof}
Let $f = \sum_{i = 1}^r \lambda_i l_i^d$ a rank decomposition of $f$ and $\z_i \in \k^n$ be the points corresponding to the linear forms $l_i$.
For every $i$ we have $l_i^d = \sum_{|\a| \leq d} {d \choose \a} \x^{\a} \z_i^{\a}$ and $\{ {d \choose \a} \x^{\a} \}_{|\a| \leq d}$ is a $\k$-basis of $R_{\leq d}$. Since $\{l_i^d\}_{i \in \{1,...,r\}}$ are $\k$-linearly independent then the matrix
\begin{equation*}
    ( \z_i^{\a} )_{ 1 \leq i \leq r, |\a| \leq d }
\end{equation*}
has rank $r$. Hence, we may consider $\{\a_j\}_{j \in \{1, \dots, r\}}$ defining a principal $r \times r$ minor
\begin{equation*}
    M = ( \z_i^{\a_j} )_{ 1 \leq i,j \leq r }.
\end{equation*}
We prove that $B = \{[x^{\a_j}]\}_{j \in \{1, \dots, r\}}$ is a $\k$-basis of $\A_{\Lambda}$, from which the statement follows.
Since $|B| = r = \dim_{\k} \A_{\Lambda}$ (by Theorem \ref{PbReformulation}), it is sufficient to prove that elements of $B$ are $\k$-linearly independent in $\A_{\Lambda}$. If $\sum_{i = 1}^r \lambda_i x^{\a_i} \in I_{\Lambda}$ for some $\{\lambda_i\}_{i \in \{1, \dots, r\}} \se \k$ then, since $\VV( I_{\Lambda} ) = \{\z_1, \dots, \z_r\}$, we have
\begin{equation*}
    \begin{cases}
    \1_{\z_1}( \sum_{i = 1}^r \lambda_i x^{\a_i} ) &= 0, \\
    &\vdots \\
    \1_{\z_r}( \sum_{i = 1}^r \lambda_i x^{\a_i} ) &= 0,
    \end{cases} \implies
    \left( \begin{array}{ccc}
    \z_1^{\a_1} & \dots & \z_1^{\a_r} \\
    \vdots &   & \vdots \\
    \z_r^{\a_1} & \dots & \z_r^{\a_r} \end{array}\right)
    \left( \begin{array}{c}
    \lambda_1 \\
    \vdots \\
    \lambda_r \end{array}\right) = \0.
\end{equation*}
Therefore, since $M$ is invertible,  $\lambda_1 = \lambda_2 = \dots = \lambda_r = 0$.
\end{proof}

We would also like to search for a $B$ that allows us to read all the coefficients we need to reconstruct the linear forms involved in the decomposition of $f$.
We prove that it is always possible if we use \emph{essential variables}.

\begin{defn} Let $F \in R^h_d$. We define the \emph{essential number of variables} $N_{\text{ess}}(F)$ of $F$ as the minimal $m \in \N$ for which there are $\{y_1, \dots, y_m\} \se R^h_{\leq 1}$ such that $F \in \k[y_1, \dots, y_m]$.
Every such a minimal set $\{y_1, \dots, y_m\}$ is called a set of \emph{essential variables} of $F$.
\end{defn}

In the literature the essential number of variables of $F$ is also known as its \emph{concise dimension} (eg. cf. \cite{land}).

\begin{defn}\label{first:catalecticant} Let $F \in R^h_d$ and let $\{M_1, \dots, M_N\}$ be the standard monomial $\k$-basis of $R^h_{d-1}$. Thus, for every $i \in \{0, 1, \dots, n\}$ there are uniquely determined constants $\{c_{ij}\}_{j \in \{1, \dots, N\}} \se \k$ such that
\begin{equation*}
\de_i(F) = c_{i1} M_1 + \dots + c_{iN} M_N.
\end{equation*}
We define the \emph{first catalecticant matrix} $\C_F \in M_{n+1, N}(\k)$ of $F$ as
\begin{equation*}
(\C_F)_{ij} = c_{ij}.
\end{equation*}
\end{defn}

The following proposition is probably classically known, but we refer to \cite[Proposition 1]{ReducingVariables} for a proof of it.

\begin{prop}\label{EssVar} Let $F \in R^h_d$. Then
\begin{equation*}
N_{ess}(F) = \rk(\C_F).
\end{equation*}
Besides, any basis of the $\k$-vector space $\langle D(\d)(F) \ | \ D \in R^h_{d-1} \rangle_{\k}$ is a set of essential variables of $F$.
\end{prop}

\begin{defn} Let $B \se R$ be a set of monomials. We say that $B$ is a \emph{staircase} if for every $\a \in \N^n$ and every $i \in \{1, \dots, n\}$ then $\x^{\a}x_i \in B$ implies $\x^{\a} \in B$.
Moreover, if $B$ also contains all the degree one monomials then we say it is a \emph{complete staircase}.
\end{defn}

\begin{thm} \label{CompleteStaircase}
Let $F \in R^h_d$ such that $\{x_0, x_1, \dots, x_n\}$ is a set of essential variables of $F$ and let $\Lambda \in R^*$ be an extension of $f^* \in R_{\leq d}^*$.
Then there is a monomial basis $B$ of $\A_{\Lambda}$ such that $B$ is a complete staircase with elements of degree at most $d$.
\end{thm}
\proof Let us consider a basis $B$ of $\A_{\Lambda}$ made of representatives of degree not greater than $d$, as in Proposition \ref{Degd}.
Let $G$ be a Gr\"obner basis of $I_{\Lambda}$ with respect to a graded order on $R$.
By \cite[Chapter 5, Section 3, Proposition 4]{IdealVarieties} reducing $B$ with respect to $G$ we obtain a staircase basis of $\A_{\Lambda}$. Since the considered order is graded, elements of $B$ still have degree bounded by $d$.

We now check that such a staircase may also be chosen complete. Let us assume by contradiction that a variable $x_j$ never occurs in the representatives of $G$, then by \cite[Chapter 5, Section 3, Proposition 1]{IdealVarieties} its reminder $\overline{x_j}^G$ obtained dividing by $G$ is a $\k$-linear combination of monomials in $B$, therefore since the order is graded we have a linear relation
\begin{equation*}
   l = \lambda_0 + \lambda_1 x_1 + \dots + \lambda_n x_n \in I_{\Lambda}.
\end{equation*}
We define $D = \lambda_0 \de_0 + \lambda_1 \de_1 + \dots + \lambda_n \de_n$ and prove that $D(F) = 0$.
In fact, a straightforward calculation shows that the coefficient of $\x^{(\a_0, \a_1, \dots, \a_n)}$ in $D(F)$ is equal to $f^*\big(l \cdot \x^{(\a_1, \dots, \a_n)}\big)$. However, this quantity is always zero because
\begin{equation*}
f^*\big(l \cdot \x^{(\a_1, \dots, \a_n)}\big) = l \star f^*\big( \x^{(\a_1, \dots, \a_n)}\big) = H_{\Lambda}(l)\big(\x^{(\a_1, \dots, \a_n)}\big) = 0.
\end{equation*}

By Proposition \ref{EssVar}, $D(F) = 0$ implies that there is a non-trivial vanishing combination between the lines of $\C_F$, contradicting $N_{ess}(F) = n+1$.
\endproof

Thus, by Theorem \ref{CompleteStaircase} we can limit our research to bases $B$ in the set
\begin{equation*}
\B_d = \{B \se R_{\leq d} \ | \ B \text{ is a complete staircase}\}.
\end{equation*}

\subsection{Minimal Waring rank to test}\label{Sec:minr}

In this section we determine the first $r$ to test in order to find a Waring decomposition.
We define the ${ n + \lceil \frac{d}{2} \rceil \choose n } \times { n + \lfloor \frac{d}{2} \rfloor \choose n }$ matrix
\begin{equation*}
\rH_{f^*} = \big( f^*(\x^{\a + \b}) \big)_{\substack{|\a| \leq \lceil d/2 \rceil\\ |\b| \leq \lfloor d/2 \rfloor}}.
\end{equation*}
For every $\Lambda \in R^*$ extending $f^* \in R_{\leq d}^*$, the matrix $\rH_{f^*}$ is the largest numerical submatrix of $\H_{\Lambda}(\mathbf{h})$, namely the largest submatrix not involving any variables $h_{\a}$.

The following proposition is actually \cite[Section 5.4]{IK}.

\begin{prop} \label{LowerRankBound}If $F \in R^h_d$ has  Waring rank $r$, then $\rk \rH_{f^*} \leq r$.
\end{prop}
\proof Let $\Lambda = \sum_{i = 1}^r \lambda_i \1_{l_i} \in R^*$ extending $f^* \in R_{\leq d}^*$ such that $r$ is minimal, i.e. $r$ is the Waring rank of $F$. By Theorem \ref{PbReformulation} we have $\rk \H_{\Lambda} = r$ and since $\rH_{f^*}$ is a submatrix of $\H_{\Lambda}$ then also $\rk \rH_{f^*} \leq \rk \H_{\Lambda}$. 
\endproof

By Proposition \ref{LowerRankBound} it is sufficient to test bases $B$ with $|B| \geq \rk \rH_{f^*}$.

\subsection{Waring decomposition algorithm}\label{SymTD:algorithm}
We are now ready to state our version of the algorithm for Waring rank and decomposition. 

We require the input polynomial $F \in R^h_d$ to be written with a general set of essential variables, i.e before starting the algorithm one has to perform a change of variables for $F$ by employing a general (numerically speaking: random) basis of the vector space given by Proposition \ref{EssVar}.
\\
\begin{mdframed}[]
\begin{alg}[Waring Decomposition]\label{alg:waring}\end{alg}
\vspace{0.1cm}
\noindent\textbf{Input:} A degree $d \geq 2$ polynomial $F \in R^h_d$ written by using a general set of essential variables. \\
    \textbf{Output:} A Waring  decomposition of $F$.
\begin{enumerate}
\item \label{Dehom&Hankel} Construct the matrix $\H_{\Lambda}(\mathbf{h})$ with parameters $\mathbf{h} = \{h_{\a}\}_{\substack{\a \in \N^n\\\hspace{-0.1cm}|\a| > d}}$.
\item \label{Startingr} Set $r := \rk \rH_{f^*}$.
\item \label{MainLoop} For $B \in \B_d$ and $|B| = r$ do
\begin{itemize}
\item Find parameters $\mathbf{h}$ such that:\\
- $\det \H_{\Lambda}^B \neq 0$.\\
- The operators $(\M_i^B)^t := \H^B_{x_i \star \Lambda}(\H^B_{\Lambda})^{-1}$ commute.\\
- There are $r$ distinct eigenvectors $v_1, \dots, v_r$ common to $(\M_i^B)^t$'s.
\item If one finds such parameters then go to step \ref{FinalStep}.
\end{itemize}
\item \label{Increaser} Set $r := r + 1$ and restart step \ref{MainLoop}.
\item \label{FinalStep} Solve the linear system $F = \sum_{i = 1}^r \lambda_i (v_{i1}x_0 + \dots + v_{i(n+1)}x_n)^d$ to find the $\{\lambda_i\}_{i \in \{1, \dots, r\}} \subseteq \k$ and return the obtained decomposition of $F$.
\end{enumerate}
\end{mdframed}
\vspace{0.5cm}
We thank B. Mourrain for having pointed us out the following.
\begin{rmk}\label{rmk:RandomCombination} If we choose $\{\gamma_1, \dots, \gamma_n\} \se \k$ randomly, the common eigenvectors of $\{(\M_i^B)^t\}_{i \in \{1, \ldots, n\}}$ are eigenvectors of $\sum_{i = 1}^n \gamma_i (\M_i^B)^t$, which are simple with probability $1$. Hence, the check on common eigenvectors requires only one eigenspace computation.
\end{rmk}

\section{Algorithm advantages}\label{algo:advantages}

In this section we give some examples of actual advantages of this version of the algorithm with respect to the one given in \cite{SymTensorDec}. Moreover, we also draw attention to the motivations behind some steps of the algorithm, exhibiting what could go wrong by ignoring them.

\subsection{Essential variables}\label{section:essential:variables}

The use of essential variables is actually \emph{essential} to fully reconstruct a Waring decomposition, as shown in Theorem \ref{CompleteStaircase}. As an example of what could go wrong by not making use of essential variables, we consider
\begin{equation*}
    F = (x+y+z)^3-x^3 \in \Cx[x,y,z].
\end{equation*}
It has Waring rank 2 but $\{x,y,z\}$ is not a set of essential variables of $F$, since $\{x, y+z\}$ is. In fact, there are no complete staircases $B$ with only 2 elements, since a complete staircase must contain at least $1, y$ and $z$. Thus, the algorithm will never come up with a rank-2 decomposition.

We also notice that the problem is not related to the choice of $B$: with such an algorithm any basis made of two elements can provide us with at most two coefficients of the linear forms in $\Cx[x,y,z]^h_1$, then by using only one $B$ it is not possible to recover all the coefficients of a Waring decomposition.

\subsection{The starting $r$}\label{Scelta:r}
By Proposition \ref{LowerRankBound} we do not miss good decompositions of the given polynomial starting the algorithm with $r=\rk\rH_{f^*}$. One might think that testing smaller $r$'s (as in \cite{SymTensorDec}) is just a waste of computational power, but there are also theoretical reasons to avoid these $r$'s.
In fact, the risk of start testing small ranks is to come up with decompositions of different tensors having many monomials in common with the one that we really want to decompose but a different (smaller) Waring rank. In the algorithmic practice, this means that the SymTD algorithm exits its main loop, reaches Step \ref{FinalStep} but cannot find any solution to the final linear system. The following example portraits precisely this issue.

\begin{exmp}
Let $F = x^4 + (x+y)^4 + (x-y)^4 = 3x^4 + 12x^2y^2 + 2y^4 \in \Cx[x,y]$. The principal $4 \times 4$ minor of the generalized Hankel matrix is
\begin{equation*}
\left( \begin{array}{cccc}
3  &  0  &  2  &  0 \\
0  &  2  &  0  &  2 \\
2  &  0  &  2  &  h_5 \\
0  &  2  &  h_5 & h_6 \end{array}\right).
\end{equation*}
Let us consider $r = 2$ instead of $r = 3$ as prescribed by the algorithm. The only possible basis $B = \{1,y\}$ leads to the following multiplication matrix
\begin{equation*}
(\M_y^B)^t = \left( \begin{array}{cc}
0  &  1  \\
\frac{2}{3} & 0 \end{array}\right).
\end{equation*}
It has two distinct eigenvectors, namely $(\pm \sqrt{3/2}, 1)$. Nevertheless, the system
\begin{equation*}
    F = \lambda_1 (\sqrt{3/2}x + y)^4 + \lambda_2 (-\sqrt{3/2}x + y)^4
\end{equation*}
has no solutions. However, if we ignored the linear condition imposed by the coefficients of $y^4$ the above system would have the solution $\lambda_1 = \lambda_2 = \frac{2}{3}$.
This choice of coefficients determines the polynomial $G = 3x^4 + 12x^2y^2 + \frac{4}{3}y^4$.
As expected, since we started from $r < \rk \rH_{f^*}$ we did not use all the information of $\rH_{f^*}$ and this has translated into a Waring decomposition of another polynomial, whose Hankel matrix has many (but not every) entries in common with the one of $F$.
\end{exmp}

\subsection{The requirements on B}\label{requirementB}

Here we discuss the choice of bases $B$ as complete staircases.
\begin{defn}
Let $B \se R$ be a set of monic monomials. We say that $B$ is \emph{connected to 1} if for every $m \in B$ either $m = 1$ or there exists $i \in \{1, \dots, n\}$ and $m' \in B$ such that $m = x_i m'$.
\end{defn}
In \cite{SymTensorDec} it is asserted that we need to check bases $B$ connected to 1. Clearly every complete staircase is also connected to 1, but the converse does not hold. Since by Theorem \ref{CompleteStaircase} we know that a Waring decomposition always arises from a basis which is a complete staircase, we have restricted the research to these particular bases.

With this requirement the number of bases to test for a given rank drops dramatically. As an example, the following table shows how many such bases are there in $\Cx[x,y,z]$ depending on their size.

\begin{table}[h!] \centering
\begin{tabular}{|l|l|l|}
\hline
Size & \# Complete staircases   & \# Connected to 1 \cr  \hline
3    & 1                        & 5              \cr  \hline
4    & 3                        & 13             \cr  \hline
5    & 5                        & 35             \cr  \hline
6    & 9                        & 96             \cr  \hline
7    & 13                       & 267            \cr  \hline
\end{tabular}
\end{table}


Moreover, the average degree of monomials inside a complete staircase is lower than the average degree inside bases connected to 1, which translates into a fewer occurrences of variables in the considered matrices. Since finding good values for the $\mathbf{h}$'s is the most computationally demanding operation performed by the algorithm, we certainly want to avoid it as much as possible.

\medskip 

Another advantage of considering basis which is a complete staircase rather than connected to one is instructively enlightened by the following example. We thank A. Iarrobino for having pointed it out to us. 
 
\begin{exmp}[Perazzo's cubic \cite{Pe}]\label{ex:Perazzo}
Let $F= xu^3+yuv^2+zu^2v$. The partial derivatives $\partial_x(F)=u^{3}$, $\partial_y (F)=uv^{2}$ and $\partial_z(F)=u^{2}v$ are algebraically dependent, so by Gordan-Noether criterion \cite{GN} the Hessian of $F$ is $0$. By the Maeno-Watanabe criterion \cite{MW} this implies that the $5\times 5$ principal minor of the first catalecticant matrix (cf. Definition \ref{first:catalecticant}) does not have maximal rank for any choices of variables, regardless $\{x,y,z,u,v\}$ is a proper set of essential variables for $F$.
By means of our algorithm this argument shows that Perazzo's cubic has Waring rank strictly greater than $5$ without even testing it. However, should one straightforwardly apply the first version of this algorithm, 867 useless bases would be checked before reaching the same conclusion.

Optimization aside, this also shows that additional theoretical information about bases might well help out while decomposing specific tensors.
\end{exmp} 

\subsection{Looking at eigenvectors}\label{section:Looking:eigenvectors}

In this section we stress the importance of Corollary \ref{RadicalDistinctEigenvectors}: checking whether $I_{\Lambda}$ is radical is actually equivalent to verifying the condition on common eigenvectors, so that every not equivalent test would inevitably carry some issues. As an example, asking for multiplication matrices to have simple eigenvalues (as in \cite{SymTensorDec}) is a sufficient condition in order to have a radical ideal $I_\Lambda$, but it is not necessary if we search for a minimal decomposition. In fact, there are instances where this condition misses good Waring decompositions, such as the following. Let us consider
\begin{equation*}
    F = (x+y)^3+(x+z)^3+(x+y+z)^3 \in \Cx[x,y,z].
\end{equation*}

There is only one $B$ with three elements to test, namely $B = \{1,y,z\}$, which gives the following multiplication matrices
\begin{equation*}
    (\M_{y}^B)^t =
    \left( \begin{array}{ccc}
    0 & 1 & 0 \\
    0 & 1 & 0 \\
    -1 & 1 & 1 \end{array}\right), \ \
    (\M_{z}^B)^t =
    \left( \begin{array}{ccc}
    0 & 0 & 1 \\
    -1 & 1 & 1 \\
    0 & 0 & 1 \end{array}\right).
\end{equation*} 
Should we check their eigenvalues, we would conclude that the Waring rank of $F$ is at least $4$ because none of them have only simple eigenvalues. However, they have exactly three common eigenvectors
\begin{equation*}
    \begin{array}{c}
    1 \ \rightarrow \\
    y \ \rightarrow \\
    z \ \rightarrow \end{array} \ \ \ \
    \left( \begin{array}{c}
    1 \\
    0 \\
    1\end{array}\right), \ \
    \left( \begin{array}{c}
    1 \\
    1 \\
    0 \end{array}\right), \ \
    \left( \begin{array}{c}
    1 \\
    1 \\
    1 \end{array}\right),
\end{equation*} 
which in fact give rise to a correct Waring decomposition of $F$. Nevertheless, we should mention that these cases almost never occur by using a general set of variables as required by the input of our algorithm, then one might prefer checking the eigenvalues to speed the algorithm up.

\section{The tangential case}\label{Section:Tangential}

\subsection{Generalizing previous results}

Let $l,g \in R_{\leq 1}$ be affine linear forms. In this section we show how a slight generalization of the algorithm proposed in Section \ref{SymTD:algorithm} may produce decompositions of degree $d$ polynomials made of pieces of the form $l^{d-1}g$, namely by using points on the tangent space of the Veronese variety. The Waring decomposition arises from the particular case $g = l$.

First, we need to generalize Proposition \ref{Prop:DualLinearAffine}.

\begin{prop}\label{Prop:TangentialCase}
Let $l = 1 + l_1x_1 + \dots + l_nx_n \in R_{\leq 1}$ and $g = 1 + g_1x_1 + \dots + g_nx_n \in R_{\leq 1}$. For every $d \in \Z_{\geq 1}$ we have
\begin{equation*}
    \tau( l^{d-1} g ) = \1_{l} + \frac{1}{d} \1_{l} \circ \left[\sum_{i = 1}^n (g_i - l_i) \frac{\de}{\de x_i}\right] \in R_{\leq d}^*.
\end{equation*}
\end{prop}

\proof For every $f = \sum_{|\a| \leq d} f_{\a}\x^{\a} \in R_{\leq d}$ a straightforward computation shows that
\begin{equation*}
    (l^{d-1} g)^* (f) = \left(\frac{1}{d} \1_{l} \circ \sum_{i = 1}^n g_i \frac{\de}{\de x_i}\right) (f) + \frac{1}{d} \1_{l} \left( \sum_{|\a| \leq d} (d - |\a|) f_{\a} \x^{\a} \right).
\end{equation*}
Now we recall the Euler's Homogeneous Function Theorem: for every homogeneous function $F$ of order $d$ in $n+1$ variables, we have
\begin{equation*}
    \sum_{i = 0}^n x_i \frac{\de F}{\de x_i} = d F(\x).
\end{equation*}
We use as $F \in R_d^h$ the degree $d$ homogenization of $f$, then we dehomogenize the above formula with respect to $x_0$ obtaining
\begin{equation*}
    \sum_{|\a| \leq d} (d - |\a|) f_{\a} \x^{\a} = d f - \sum_{i = 1}^n x_i \frac{\de f}{\de x_i}.
\end{equation*}
Therefore, we conclude
\begin{align*}
    (l^{d-1} g)^* (f) &= \left(\frac{1}{d} \1_{l} \circ \sum_{i = 1}^n g_i \frac{\de}{\de x_i}\right) (f) + \frac{1}{d} \1_{l} \left( d f - \sum_{i = 1}^n x_i \frac{\de f}{\de x_i} \right) \\
    &= \left( \1_{l} + \frac{1}{d} \1_{l} \circ \sum_{i = 1}^n (g_i - l_i) \frac{\de}{\de x_i}\right) (f),
\end{align*}
which proves the statement.
\endproof

We use Proposition \ref{Prop:TangentialCase} to read the coefficients of these forms from the multiplication operators.

\begin{thm} \label{Thm:TangentialDecomposition}
Let $l = 1 + l_1x_1 + \dots + l_nx_n \in R_{\leq 1}$ and $g = 1 + g_1x_1 + \dots + g_nx_n \in R_{\leq 1}$. Let $\Lambda \in R^*$ such that $I_{\Lambda}$ is zero-dimensional and $\Gamma \in R^*$ such that $\Gamma|_{R_{\leq d}} = (l^{d-1}g)^* \in R_{\leq d}^*$ and $\Gamma \in I_{\Lambda}^{\perp}$.
Let also $\{ M^t_{x_i} \}_{i \in \{1, \dots, n\}}$ be the multiplication-by-$x_i$ operators on $\A_{\Lambda}^*$.
Then
\begin{itemize}
    \item for the $j$'s such that  $g_j = l_j$ we have that $\Gamma$ is an eigenvector of $M_{x_j}^t$;
    \item for the $j$'s such that $g_j \neq l_j$ we have that $\Gamma$ is a generalized eigenvector of rank 2 of $M_{x_j}^t$ and the chain it generates is $\{\Gamma, \1_{l}\}$.
\end{itemize}
\end{thm}

\proof If $\VV(I_{\Lambda}) = \{\z_i\}_{i \in \{1, \dots, e\}}$ by Theorem \ref{Thm:Structure} we have $\Gamma = \sum_{i = 1}^e \1_{\z_i} \circ p_i(\delta)$ and by Proposition \ref{Prop:TangentialCase} we have $\Gamma|_{R_{\leq d}} = \1_{l} \circ \left[1 + \sum_{i = 1}^n \frac{(g_i - l_i)}{d} \frac{\de}{\de x_i} \right] \in R_{\leq d}^*$. Since $\{ [\x^{\a}(\delta)]_{\zeta} \}_{\a \in \N^n}$ is a $\k$-basis of $R^*$ \cite[Chapter 7]{FrenchRef} we conclude that $\1_{l} = \1_{\z_k}$ for some $k \in \{1, \dots, e\}$ and that, up to scalars, we have
\begin{equation*}
    \Gamma = \1_{l} + \frac{1}{d} \1_{l} \circ \sum_{i = 1}^n (g_i - l_i) \frac{\de}{\de x_i} \in R^*.
\end{equation*}
By the derivation of a product rule we have that for every $j \in \{1, \dots, n\}$
\begin{align*}
    M_{x_j}^t \Gamma &= x_j \star \Gamma = l_j\1_{l} + \frac{l_j}{d} \1_{l} \circ \sum_{i = 1}^n (g_i - l_i) \frac{\de}{\de x_i} + \frac{1}{d} \sum_{i = 1}^n (g_i - l_i) \frac{\de}{\de x_i}(x_j) \1_{l}\\
    &= l_j \Gamma + \frac{g_j - l_j}{d} \1_l.
\end{align*}
If $g_j = l_j$ then $\Gamma$ is an eigenvector of $M_{x_j}^t$, whereas if $g_j \neq l_j$ then $(M_{x_j}^t - l_j \1) (\Gamma)$ is a non-zero multiple of $\1_l$, which is an eigenvector for every $M_{x_j}^t$ by Theorem \ref{Eigenvectors}. This means precisely that $\Gamma$ is a generalized eigenvector of rank 2 of $M_{x_j}^t$ and that its chain is $\{\Gamma, \1_{l}\}$.
\endproof

Theorem \ref{Thm:TangentialDecomposition} shows that we may find decompositions of a given degree $d$ polynomial involving pieces of type $l^{d-1}g$ (possibly with $l = g$) by looking at the generalized eigenvectors of multiplication matrices. However, if we want to minimize the number of considered linear forms, we need to count twice the pieces where $l \neq g$. It motivates the following definition.

\begin{defn}
Let $F \in R_d^h$. We define its \emph{tangential rank} as the minimal $r \in \N$ such that there exist two integers $k \leq s$ with 
 $s+k = r$, scalars $\{\lambda_1, \dots, \lambda_s\} \se \k$ and linear forms $\{L_1, \dots, L_r\} \se R_1^h$ such that
\begin{equation*}
    F = \sum_{i = 1}^k \lambda_i L_i^{d-1}L_{s+i} + \sum_{i = k+1}^s \lambda_i L_i^d.
\end{equation*}
Such a decomposition for which $r$ is minimal is referred to as a \emph{tangential decomposition} of $F$.
\end{defn}

For a reader who is not familiar with Algebraic Geometry, we have chosen the denomination \emph{tangential decomposition} because it corresponds to a decomposition of $[F]$ in terms of points on the tangential variety of a Veronese variety. Remark that the tangential rank is not the minimum number of points on such a variety occurring in a tangential decomposition. Indeed it can  be identified with the minimal length of a 0-dimensional scheme contained in the tangential variety of the Veronese variety whose span contains $[F]$.

In the next section we adapt the algorithm of Section \ref{SymTD:algorithm} in order to detect tangential decompositions.

\begin{rmk} \label{Rmk:degreeofB}
For the tangential case we do not currently have a result such as Proposition \ref{Degd}, which in the Waring case ensures us that the basis may be chosen of degree bounded by $d = \deg F$.
This means that we are are forced to require our algorithm to search inside $\B_{r}$ instead of $\B_{d}$, even if we have never encountered cases where $\B_{d}$ does not suffice.
\end{rmk}

\subsection{Tangential decomposition algorithm} \label{Sec:TangentialAlgorithm}
\begin{mdframed}[]
\begin{alg}[Tangential Decomposition]\label{alg:tg}\end{alg}
\vspace{0.1cm}
\noindent\textbf{Input:} A degree $d \geq 2$ polynomial $F \in R^h_d$ written by using a general set of essential variables. \\
    \textbf{Output:} A tangential decomposition of $F$.
\begin{enumerate}
\item \label{Dehom&Hankel2} Construct the matrix $\H_{\Lambda}(\mathbf{h})$ with parameters $\mathbf{h} = \{h_{\a}\}_{\substack{\a \in \N^n\\\hspace{-0.1cm}|\a| > d}}$.
\item \label{Startingr2} Set $r := \rk \rH_{f^*}$.
\item \label{MainLoop2} For $B \in \B_{r}$  and $|B| = r$ do
\begin{itemize}
\item Find parameters $\mathbf{h}$ such that:\\
- $\det \H_{\Lambda}^B \neq 0$.\\
- The operators $\M_i^B := \H^B_{x_i \star \Lambda}(\H^B_{\Lambda})^{-1}$ commute.\\
- There are $\frac{r}{2} \leq s \leq r$ distinct eigenvectors $v_1, \dots, v_s$ common to all the $\M_i^B$'s. \\
- There are $r-s$ distinct generalized of rank up to $2$ eigenvectors $v_{s+1}, \ldots , v_r$ common to all the $\M_i^B$'s such that
\begin{itemize}
    \item they have rank $2$ for at least one $\M_i^B$,
    \item when they have rank $2$, their chain is always $\{v_{s+i}, v_i\}$.
\end{itemize}
\item If one finds such parameters then go to step (\ref{FinalStep2}).
\end{itemize}
\item \label{Increaser2} Set $r := r + 1$ and restart step (\ref{MainLoop2}).
\item \label{FinalStep2} Define $\{L_i = (v_i)_1x_0 + \dots + (v_i)_{n+1}x_n\}_{i \in \{1, \dots, r\}}$,  find $\lambda_1, \dots, \lambda_r \in \k$ by solving the linear system 
\begin{equation*}
    F = \sum_{i = 1}^{r-s} L_i^{d-1}(\lambda_i L_i + \lambda_{s+i} L_{s+i}) + \sum_{i = r-s+1}^s \lambda_i L_i^d
\end{equation*}
and return the obtained decomposition of $F$.
\end{enumerate}
\end{mdframed}

\subsection{Some examples}

Here we perform the tangential decomposition algorithm of Section \ref{Sec:TangentialAlgorithm} on some polynomials, detailing the crucial steps.

\begin{exmp}
We begin with an easy case, where there is no need to fill the generalized Hankel matrix.

Let $F \in \Cx[x,y,z]$ be the homogeneous of degree $5$ polynomial given by

\begin{footnotesize}\begin{align*}
    F =& \ x^5 + 32x^4y - 36x^4z - 62x^3y^2 + 220x^3yz - 154x^3z^2 + 172x^2y^3 - 744x^2y^2z \\
    &+ 1140x^2yz^2 - 556x^2z^3 - 157xy^4 + 948xy^3z - 2118xy^2z^2 + 2132xyz^3 - 799xz^4\\
    &+ 64y^5 - 482y^4z + 1448y^3z^2 - 2172y^2z^3 + 1628yz^4 - 488z^5.
\end{align*}\end{footnotesize}

We dehomogenize $F$ by $x = 1$ and construct the generalized Hankel matrix $\H_{\Lambda}(\mathbf{h})$. Below we include the $9 \times 9$ principal minor.

\begin{tiny}\[ \left( \begin{array}{c|cccccccccc}
  & 1 & y & z & y^2 & yz & z^2 & y^3 & y^2z & yz^2 \\
\hline
1  & 1    &32/5   &-36/5   &-31/5      &11   &-77/5    &86/5  &-124/5      &38  \\
y  & 32/5   &-31/5      &11    &86/5  &-124/5      &38  &-157/5   &237/5  &-353/5  \\
z  & -36/5      &11   &-77/5  &-124/5     &38  &-278/5   &237/5  &-353/5   &533/5  \\
y^2  &  -31/5    &86/5  &-124/5  &-157/5   &237/5  &-353/5      &64  &-482/5   &724/5  \\
yz  & 11  &-124/5      &38   &237/5  &-353/5   &533/5  &-482/5   &724/5 &-1086/5 \\
z^2  & -77/5      &38  &-278/5  &-353/5   &533/5  &-799/5   &724/5 &-1086/5  &1628/5   \\
y^3  &  86/5  &-157/5   &237/5      &64  &-482/5   &724/5    &h_{6,0}    &h_{5,1}    &h_{4,2}  \\
y^2z & -124/5   &237/5  &-353/5  &-482/5   &724/5 &-1086/5    &h_{5,1}    &h_{4,2}    &h_{3,3}  \\
yz^2  & 38  &-353/5   &533/5   &724/5 &-1086/5  &1628/5    &h_{4,2}    &h_{3,3}    &h_{2,4}  \\
\end{array} \right).\]\end{tiny}

The rank of the largest numerical submatrix is $5$, hence we start from $r = 5$.
We pick $B = \{1,y,z,y^2,yz\}$, check that $\H_{\Lambda}^B$ is invertible and define the multiplication operators
\begin{footnotesize}\begin{equation*}
     (\M_{y}^B)^t =
\left( \begin{array}{ccccc}
0 & 1 & 0 & 0 & 0 \\
0 & 0 & 0 & 1 & 0 \\
0 & 0 & 0 & 0 & 1 \\
-2 & 3 & 0 & 0 & 0 \\
 3 & -6 & -1 & 3 & 2 \end{array}\right), \ \
(\M_{z}^B)^t =
\left( \begin{array}{ccccc}
0 & 0 & 1 & 0 & 0 \\
0 & 0 & 0 & 0 & 1 \\
20/9 & -31/9  &  -1 & 20/9  &   1 \\
3  &  -6  &  -1   &  3  &   2 \\
-13/9 & 26/9  &  -1 & -4/9  &  1 \end{array}\right).
\end{equation*}
\end{footnotesize}They commute and their eigenspaces are
\begin{scriptsize}
\begin{equation*}
    \left\langle \left( \begin{array}{c}1 \\ -2 \\ 3 \\ 4 \\ -6 \end{array}\right) \right\rangle,
    \left\langle \left( \begin{array}{c}1 \\ 1 \\ 0 \\ 1 \\ 0 \end{array}\right), \left( \begin{array}{c}0 \\ 0 \\ 1 \\ 0 \\ 1 \end{array}\right) \right\rangle
    \begin{normalsize}\text{and}\end{normalsize} \left\langle \left( \begin{array}{c}1 \\ -2 \\ 3 \\ 4 \\ -6 \end{array}\right) \right\rangle,
    \left\langle \left( \begin{array}{c}1 \\ 1 \\ 1 \\ 1 \\ 1 \end{array}\right)\right\rangle,  \left\langle \left( \begin{array}{c}1 \\ 0 \\ -1 \\ -1 \\ 0 \end{array}\right), \left( \begin{array}{c}0 \\ 1 \\ 0 \\ 2 \\ -1 \end{array}\right) \right\rangle
\end{equation*}
\end{scriptsize}respectively. We see that there are $s = 3$ common eigenvectors, corresponding to the linear forms $l_1 = 1 + y + z$, $l_2 = 1 + y -z$ and $l_3 = 1 -2y + 3z$. Now we look at rank $\leq 2$ eigenvectors: those of $(\M_y^B)^t$ are
\begin{scriptsize}
\begin{equation*}
    \left\langle \left( \begin{array}{c}1 \\ -2 \\ 3 \\ 4 \\ -6 \end{array}\right) \right\rangle,
    \left\langle \left( \begin{array}{c}1 \\ 1 \\ 0 \\ 1 \\ 0 \end{array}\right),
    \left( \begin{array}{c}0 \\ 0 \\ 1 \\ 0 \\ 1 \end{array}\right)
    \left( \begin{array}{c}0 \\ 1 \\ 0 \\ 2 \\ 0 \end{array}\right)
    \left( \begin{array}{c}0 \\ 0 \\ 0 \\ 0 \\ 1 \end{array}\right)\right\rangle,
\end{equation*}
\end{scriptsize}whereas those of $(\M_z^B)^t$ are
\begin{scriptsize}\begin{equation*}
   \left\langle \left( \begin{array}{c}1 \\ -2 \\ 3 \\ 4 \\ -6 \end{array}\right) \right\rangle,
    \left\langle \left( \begin{array}{c}1 \\ 1 \\ 1 \\ 1 \\ 1 \end{array}\right),
    \left( \begin{array}{c}1 \\ 0 \\ 0 \\ -1 \\ -1 \end{array}\right) \right\rangle,
    \left\langle \left( \begin{array}{c}1 \\ 0 \\ -1 \\ -1 \\ 0 \end{array}\right),
    \left( \begin{array}{c}0 \\ 1 \\ 0 \\ 2 \\ -1 \end{array}\right) \right\rangle.
\end{equation*}
\end{scriptsize}

\noindent The vector $(1, 0, 0, -1, -1)$ is a rank-2 eigenvector for both $(\M_y^B)^t$ and $(\M_z^B)^t$, with a chain ending in $\langle (1, 1, 1, 1, 1) \rangle$. Thus, the linear form $l_4 = 1$ is paired with $l_1$ in the tangential decomposition.
We notice that according to Theorem \ref{Thm:TangentialDecomposition}, since both the $y$ and the $z$ coefficients of $l_4$ are different from the correspondent coefficients of $l_1$, for both the multiplication matrices the vector $(1, 0, 0, -1, -1)$ appears as a rank-2 eigenvector.

On the other side, the vector $(1, 0, -1, -1, 0)$ is an eigenvector for $(\M^B_z)^t$ and a rank-2 eigenvector with chain ending in $\langle (1, 1, -1, 1, -1) \rangle$ for $(\M^B_y)^t$.
Again, we notice that $l_5 = 1 - z$ has the same $z$-coefficient of $l_2$, but a different $y$-one, as prescribed by Theorem \ref{Thm:TangentialDecomposition}.

Since we have found $2 = r-s$ generalized eigenvectors satisfying the required conditions, we obtain a tangential decomposition of $F$ by solving the linear system

\begin{footnotesize}
\begin{equation*}
    F = (x+y+z)^4\big( \lambda_1(x+y+z) + \lambda_4x \big) + (x+y-z)^4\big( \lambda_2(x+y-z) + \lambda_5(x-z) \big) + \lambda_3(x-2y+3z)^5.
\end{equation*}
\end{footnotesize}

\noindent The solution $(\lambda_1, \dots, \lambda_5) = (0, 0, -2, 1, 2)$ leads to the tangential decomposition
\begin{equation*}
    F = (x+y+z)^4(x) + 2(x+y-z)^4(x-z) - 2(x-2y+3z)^5.
\end{equation*}
Therefore the tangential rank of $F$ is 5.

\end{exmp}

\begin{exmp}
This example is meant to stress that the linear forms appearing in a tangential decomposition are not required to be different. Clearly a tangential decomposition can not contain repeated $d-1$ powers of the same linear form, but it might well contain repeated linear factors.

Let $F \in \Cx[x,y,z]$ be the homogeneous of degree $7$ polynomial given by

\begin{footnotesize}\begin{align*}
    F =& -2x^7 - 4x^6y + 92x^6z + 15x^5y^2 - 675x^5z^2 - 20x^4y^3 + 2700x^4z^3 + 15x^3y^4\\
      &- 6075x^3z^4 - 6x^2y^5 + 7290x^2z^5 + xy^6 - 3645xz^6.
\end{align*}\end{footnotesize}

We dehomogenize $F$ by $x = 1$ and construct the generalized Hankel matrix $\H_{\Lambda}(\mathbf{h})$. By using the basis $B = \{1, y, z, y^2, z^2, y^3\}$ we get the multiplication matrices

\begin{small}
\begin{equation*}
(\M_{y}^B)^t =
\left( \begin{array}{cccccc}
0 & 1 & 0 & 0 & 0 & 0 \\
0 & 0 & 0 & 1 & 0 & 0 \\
0 & 0 & 0 & 0 & 0 & 0 \\
0 & 0 & 0 & 0 & 0 & 1 \\
0 & 0 & 0 & 0 & 0 & 0 \\
0 & 0 & 0 & -1 & 0 & -2 \end{array}\right), \ \
(\M_{z}^B)^t =
\left( \begin{array}{cccccc}
0 & 0 & 1 & 0 & 0 & 0 \\
0 & 0 & 0 & 0 & 0 & 0 \\
0 & 0 & 0 & 0 & 1 & 0 \\
0 & 0 & 0 & 0 & 0 & 0 \\
0 & 9 & -9 & 18 & -6 & 9 \\
0 & 0 & 0 & 0 & 0 & 0 \end{array}\right).
\end{equation*}
\end{small}

\noindent The common rank-1 eigenvectors are $v_1 = (1, 0, 0, 0, 0, 0), v_2 = ( 1, -1,  0,  1,  0, -1)$ and $v_3 = ( 1,  0, -3,  0,  9,  0)$ whereas the generalized eigenvectors are
\begin{itemize}
    \item $v_4 = (1, 1, 1, 0, 0, 0)$, rank-2 for both $\M_{y}^B$ and $\M_{z}^B$, relative to $\langle v_1 \rangle$.
    \item $ v_5 = (1, 0, 0, -1, 0, 2)$, rank-2 for $\M_{y}^B$ relative to $\langle v_2 \rangle$ and rank-1 for $\M_{z}^B$.
    \item $v_6 = (1, 0, 0, 0, -9, 0)$, rank-2 for $\M_{z}^B$ relative to $\langle v_3 \rangle$ and rank-1 for $\M_{y}^B$.
\end{itemize}
Thus, the solution of the linear system leads to
\begin{equation*}
    F = 2x^6(x+y+z) + (x-y)^6x - 5(x-3z)^6x.
\end{equation*}
The tangential rank of $F$ is 6. We notice that the linear form $x$ appears three times, but only once as a sixth power.
\end{exmp}
\begin{exmp}
In this example we illustrate a difficult case, where the values of some $\mathbf{h}$'s have to be determined.

Let $F \in \Cx[x,y,z]$ be the homogeneous of degree $3$ polynomial defined by
\begin{align*}
    F   &= (x+y)^2(x+z)+(x-z)^2(x+y+z) \\
        &= 2x^3 + 3 x^2y + xy^2 - xz^2 + y^2z + yz^2 + z^3.
\end{align*}
We dehomogenize $F$ by $x = 1$ and construct $\H_{\Lambda}(\mathbf{h})$, which has the following $7 \times 7$ principal minor

{ \footnotesize \begin{equation*} \left( \begin{array}{c|ccccccc}
  & 1 & y & z & y^2 & yz & z^2 & y^3 \\
\hline
1  & 2 & 1 & 0 & 1/3 &0 & -1/3 & 0 \\
y  & 1 & 1/3 & 0 & 0 & 1/3 & 1/3  &h_{4,0} \\
z  & 0 & 0 & -1/3 & 1/3 & 1/3 & 1 &h_{3,1} \\
y^2  & 1/3 & 0 & 1/3 &h_{4,0}   &h_{3,1}  &h_{2,2}  &h_{5,0} \\
yz  & 0 & 1/3 & 1/3 &h_{3,1}   &h_{2,2}  &h_{1,3}  &h_{4,1} \\
z^2  & -1/3 & 1/3 & 1 &h_{2,2}   &h_{1,3}  &h_{0,4}  &h_{3,2} \\
y^3  &  0  &h_{4,0}   &h_{3,1}  &h_{5,0}  &h_{4,1}  &h_{3,2}   &h_{6,0}
\end{array} \right).\end{equation*}}

\noindent The rank of the largest numerical submatrix is $3$. However, the only choice for a basis $B$ with $r=3$ is $B = \{1,y,z\}$, which leads to the following matrices
\begin{equation*}
(\M_{y}^B)^t =
\left( \begin{array}{ccc}
0 & 1 & 0 \\
-1/3 & 1 & -1 \\
1 & -2 & -1 \end{array}\right), \ \
(\M_{z}^B)^t =
\left( \begin{array}{ccc}
0 & 0 & 1 \\
1 & -2 & -1 \\
4/3 & -3  & -3 \end{array}\right).
\end{equation*}
They do not commute. Since we have tested all the possible bases for $r = 3$ and none has worked, we start considering $r = 4$. Here we have three possible choices for $B$ and we consider $B = \{1, y, z, y^2\}$. Thus, we want
\begin{equation*}
\H_{\Lambda}^{B} = \left( \begin{array}{cccc}
 2 & 1 & 0 & 1/3 \\
1 & 1/3 & 0 & 0 \\
0 & 0 & -1/3 & 1/3 \\
1/3 & 0 & 1/3 &h_{4,0}
\end{array}\right)
\end{equation*}
to be invertible, so we pick $h_{4,0} = 0$. Afterwards, we can construct the multiplication matrices

\begin{footnotesize}\begin{align*}
(\M_{y}^B)^t &=
\left( \begin{array}{cccc}
0 & 1 & 0 & 0 \\
0 & 0 & 0 & 1 \\
\frac{3}{4}h_{3,1} + 1 & -\frac{9}{4}h_{3,1} - 2 & \frac{9}{4}h_{3,1} - 1 & \frac{9}{4}h_{3,1} \\
\frac{3}{4}h_{3,1} + \frac{3}{4}h_{5,0} & -\frac{9}{4}h_{3,1} - \frac{9}{4}h_{5,0} & -\frac{3}{4}h_{3,1} + \frac{9}{4}h_{5,0} & \frac{9}{4}h_{3,1} + \frac{9}{4}h_{5,0} \end{array}\right), \\
(\M_{z}^B)^t &=
\left( \begin{array}{cc}
0 & 0 \\
\frac{3}{4}h_{3,1} + 1 & -\frac{9}{4}h_{3,1} - 2 \\
\frac{3}{4}h_{2,2} + \frac{7}{4} &  -\frac{9}{4}h_{2,2} - \frac{17}{4} \\
\frac{9}{4}h_{3,1} + \frac{3}{4}h_{2,2} + \frac{3}{4}h_{4,1} - \frac{1}{4} & -\frac{15}{4}h_{3,1} - \frac{9}{4}h_{2,2} - \frac{9}{4}h_{4,1} + \frac{3}{4} \end{array} \right. \\
& \qquad \qquad \qquad \qquad \left. \begin{array}{cc}
 1 & 0 \\
 \frac{9}{4}h_{3,1} - 1 & \frac{9}{4}h_{3,1} \\
 \frac{9}{4}h_{2,2} - \frac{7}{4} & \frac{9}{4}h_{2,2} + \frac{5}{4}\\
 -\frac{9}{4}h_{3,1} - \frac{3}{4}h_{2,2} + \frac{9}{4}h_{4,1} + \frac{1}{4} &  -\frac{9}{4}h_{3,1} + \frac{9}{4}h_{2,2} + \frac{9}{4}h_{4,1} + \frac{1}{4} \end{array} \right).
\end{align*}
\end{footnotesize}

We observe that the choices $h_{3,1} = 0, h_{2,2} = -1, h_{5,0} = 0, h_{4,1} = 0$ make the multiplication matrices commute.  With these choices of the $\mathbf{h}$'s both $(\M_{y}^B)^t$ and $(\M_{z}^B)^t$ have two proper eigenvectors
\begin{align*}
    v_1 &= (  \sqrt{3}/6 - 1/4, \sqrt{3}/12 - 1/4, -\sqrt{3}/12, \sqrt{3}/6 ),\\
    v_2 &= ( - \sqrt{3}/6 - 1/4, - \sqrt{3}/12 - 1/4, \sqrt{3}/12, -\sqrt{3}/6 ),
\end{align*}
and two rank-2 eigenvectors
\begin{align*}
    v_3 &= ( 1/2 - \sqrt{3}/4, -\sqrt{3}/12, -\sqrt{3}/6, \sqrt{3}/6 ),\\
    v_4 &= ( 1/2 + \sqrt{3}/4, \sqrt{3}/12, \sqrt{3}/6, -\sqrt{3}/6 ).
\end{align*}

\noindent We can therefore write $F$ as

\begin{footnotesize}
\begin{align*}F =& \bigg[ (\sqrt{3}/6 - 1/4)x+ (\sqrt{3}/12 - 1/4)y+ (-\sqrt{3}/12)z \bigg]^2 \bigg[ \lambda_1\Big((\sqrt{3}/6 - 1/4)x \\
& \; \; \; \; \; \; +  (\sqrt{3}/12 - 1/4)y+ (-\sqrt{3}/12)z \Big) + \lambda_3 \Big((- \sqrt{3}/4)x+ (-\sqrt{3}/12)y+(-\sqrt{3}/6)z \Big) \bigg] \\
&+ \bigg[ (- \sqrt{3}/6 - 1/4)x+(- \sqrt{3}/12 - 1/4)y+(\sqrt{3}/12)z \bigg]^2 \bigg[ \lambda_2\bigg(
(- \sqrt{3}/6 - 1/4)x \\
&\; \; \; \; \; \;  + (- \sqrt{3}/12 - 1/4)y+(\sqrt{3}/12)z \Big) + \lambda_4 \Big((\sqrt{3}/4 + 1/2)x+(\sqrt{3}/12)y+(\sqrt{3}/6)z \Big) \bigg].
\end{align*}
\end{footnotesize}

\noindent In fact, the system has solution for
\begin{align*}
\lambda_1 &=-(16\sqrt{3}(26\sqrt{3} - 45))/((4\sqrt{3} - 7)^2(\sqrt{3} - 2)), \\
\lambda_2 &=(16(7\sqrt{3} - 12)(26\sqrt{3} - 45))/((4\sqrt{3} - 7)(\sqrt{3} - 2)), \\
\lambda_3 &=(48(41\sqrt{3} - 71))/((4\sqrt{3} - 7)^2(\sqrt{3} - 2)), \\
\lambda_4 &=-(16\sqrt{3}(26\sqrt{3} - 45)(\sqrt{3} - 1))/(4\sqrt{3} - 7).
\end{align*}
It is worth noting that the decomposition we found is different from the one we started with but the tangential rank is the same, namely $r = 4$. It is perfectly fine: in this case the decomposition is not unique and different choices of $\mathbf{h}$'s would have provided us with possibly different decompositions of $F$.
\end{exmp}

\section{The cactus rank}\label{section:cactus:rank}

\subsection{Evaluation of the cactus rank}

The \emph{cactus rank} of a homogeneous polynomial $F \in R^h_d$ is known to be the minimal length of an apolar zero-dimensional scheme to $F$. 
By \cite[Theorem 3.7]{bbm} the cactus rank of $F$ coincides with the size of a generalized decomposition of $f^*$ as mentioned in Section \ref{apolarity:section} and this is the definition that we use here. 

\begin{defn}\label{cactus:definition} Let $F \in R^h_d$. The \emph{cactus rank} of $F$ is the minimal $r \in \N$ such that there exists $\Lambda \in R^*$ extending $f^* \in R_{\leq d}^*$ with $I_{\Lambda}$ zero-dimensional ideal and $\dim_{\k} I_{\Lambda}^{\perp} = r$.
\end{defn}

Since $I_{\Lambda}^{\perp}$ and $\A_{\Lambda}^*$ are isomorphic as $\k$-vector spaces, we may use our algorithm to detect the first $r = \dim_{\k} \A_{\Lambda} = \dim_{\k} \A_{\Lambda}^*$ which allows to extend $f^*$ to a $\Lambda \in R^*$ such that $\rk H_{\Lambda} = r$. This is equivalent (by \cite[Theorem 6.2]{SymTensorDec}) to search for the minimal rank that the filled matrix $\H_{\Lambda}(\mathbf{h})$ may have in order to make the operators $(\M_i^B)^t$ commute. Once we have found these commuting operators we can immediately read the $\1_{\z_i}$'s appearing in a generalized decomposition of $f^*$ by Theorem \ref{Eigenvectors} as their common rank-1 eigenvectors.

Moreover, it is also possible to recover the multiplicities of these $\z_i$'s by making use of the following theorem.

\begin{thm}\label{Thm:CactusStructure}
Let $\Lambda \in R^*$ and $\Lambda = \sum_{i = 1}^d \1_{\z_i} \circ p_i(\d)$ be a generalized decomposition of $\Lambda$. Then for every $i \in \{1, \dots, d\}$ and every $\a \in \N^n$ the element
\begin{equation*}
    \1_{\z_i} \circ (\de^{\a} p_i)(\d) \in \A_{\Lambda}^*
\end{equation*}
is either the zero map or a generalized eigenvector common to every $M_{x_j}^t$ with eigenvalue $(\z_i)_j$.
Moreover, the chain of $\1_{\z_i} \circ (\de^{\a} p_i)(\d)$ relative to $M_{x_j}^t$ is
\begin{equation*}
    \{ \1_{\z_i} \circ (\de^{\a} p_i)(\d), \1_{\z_i} \circ (\de_j\de^{\a} p_i)(\d),\1_{\z_i} \circ (\de_j^2\de^{\a} p_i)(\d), \dots \},
\end{equation*}
until a proper rank-1 eigenvector is reached.
\end{thm}
\proof For every polynomial $q \in R$ and indices $i \in \{1, \dots, d\}$ and $j \in \{1, \dots, n\}$, a repeated use of the chain rule for partial derivatives shows that
\begin{equation*}
    (\de^{\a} p_i) (q x_j) = \Big( (\de_j \de^{\a} p_i)(\d) + x_j (\de^{\a} p_i)(\d) \Big) (q).
\end{equation*}
Therefore, we have
\begin{equation*}
    M_{x_j}^t \big( \1_{\z_i} \circ (\de^{\a} p_i)(\d) \big) = \1_{\z_i} \circ (\de_j \de^{\a} p_i)(\d) + (\z_i)_j \1_{\z_i} \circ (\de^{\a} p_i)(\d),
\end{equation*}
which implies inductively on $n \in \N$ that
\begin{equation*}
    \Big( M_{x_j}^t - (\z_i)_j \Big)^n \big( \1_{\z_i} \circ (\de^{\a} p_i)(\d) \big) = \1_{\z_i} \circ (\de_j^n \de^{\a} p_i)(\d).
\end{equation*}
This proves that $\1_{\z_i} \circ (\de^{\a} p_i)(\d)$ is a generalized eigenvector with eigenvalue $(\z_i)_j$ of rank $\deg_{j} \de^{\a} p_i + 1$, since its chain is obtained by repeatedly differentiating the differential polynomial $\de^{\a}p_i$ with respect to the $j$-th variable.
\endproof

\begin{cor}\label{corMu}
Let $\Lambda \in R^*$ and $\Lambda = \sum_{i = 1}^d \1_{\z_i} \circ p_i(\d)$ be a generalized decomposition of $\Lambda$. Let $V^j[\mu]$ be the generalized eigenspace of $M_{x_j}^t$ relative to the eigenvalue $\mu$, then for every $i \in \{1, \dots, d\}$ the multiplicity of $\1_{\z_i}$ is given by
\begin{equation*}
    \textnormal{mult} \, \1_{\z_i} = \dim_{\k} \cap_{j = 1}^n V^j [(\z_i)_j].
\end{equation*}
\end{cor}
\proof Since $\textnormal{mult} \, \1_{\z_i} = \dim_{\k} \langle \{ \1_{\z_i} \circ (\de^{\a} p_i)(\d)\}_{\a \in \N^n} \rangle_{\k}$, by Theorem \ref{Thm:CactusStructure} we have
\begin{equation*}
    \textnormal{mult} \, \1_{\z_i} \leq \dim_{\k} \cap_{j = 1}^n V^j [(\z_i)_j].
\end{equation*}
However, we know that the sum of all the multiplicities, say $r$, is also the dimension of $\A_{\Lambda}^*$. 
Then we have
\begin{equation*}
    r = \sum_{i = 1}^d \textnormal{mult} \, \1_{\z_i} \leq \sum_{i = 1}^d \dim_{\k} \bigcap_{j = 1}^n V^j [(\z_i)_j].
\end{equation*}

We observe that the spaces $\{ \cap_{j = 1}^n V^j [(\z_i)_j] \}_{i \in \{1, \dots, d\}}$ have trivial intersection. In fact, if $v \neq 0$ is a non-zero vector such that $v \in V^j [(\z_{i})_j]$ but also $v \in V^j [(\z_{k})_j]$ then the eigenvalue of $v$ with respect to $M_{x_j}^t$ is $(\z_{i})_j = (\z_{k})_j$. This cannot happen for every $j$, since $\z_{i} \neq \z_{k}$ by the minimality of the generalized decomposition.

Therefore, we have
\begin{equation*}
    \sum_{i = 1}^d \dim_{\k} \bigcap_{j = 1}^n V^j [(\z_i)_j] = \dim_{\k} \bigoplus_{i = 1}^d \bigcap_{j = 1}^n V^j [(\z_i)_j] \leq \dim_{\k} \bigcap_{j = 1}^n \bigoplus_{i = 1}^d V^j [(\z_i)_j],
\end{equation*}
but the sum of some generalized eigespaces cannot exceed the entire space, hence
\begin{equation*}
    \dim_{\k} \bigcap_{j = 1}^n \bigoplus_{i = 1}^d V^j [(\z_i)_j] \leq \dim_{\k} \bigcap_{j = 1}^n \A_{\Lambda}^* = \dim_{\k} \A_{\Lambda}^* = r.
\end{equation*}
By collecting the above relations we conclude that they are all equalities, so in particular for every $i \in \{1, \dots, d\}$ we have $\textnormal{mult} \, \1_{\z_i} = \dim_{\k} \cap_{j = 1}^n V^j [(\z_i)_j]$ .
\endproof

The above Theorem \ref{Thm:CactusStructure} and Corollary \ref{corMu} could be also considered as consequences of \cite[Thm. 3.1 and Prop 3.4]{Mu} and \cite[Prop. 3.8]{Mu} (respectively) after having glued together those results and reinterpreted them in the context of symmetric tensors. Since our proofs are short, constructive and well contextualized in the language of the present paper we have preferred to show them instead of referring as consequences of \cite{Mu}, whose link may be not straightforward.

\begin{cor}
Let $F\in R^h_d$ and $\Lambda \in R^*$ be an extension of $f^*\in R^*_{\leq d}$ with generalized decomposition $\Lambda = \sum_{i = 1}^s \1_{\z_i} \circ p_i(\d)$.
For every $j \in \{1, \dots, n\}$ let also $\{V^j[(\z_i)_j]\}_{i \in \{1, \dots, s\}}$ be the generalized eigenspaces of $M_{x_j}^t$ on $\A_{\Lambda}^*$ relative to the eigenvalues $(\z_i)_j$'s.
Then we have
$$F=\sum_{i=1}^sL_i^{d-k_i+1}N_i$$ 
where $L_i = x_0 + (\z_i)_1x_1 + \cdots + (\z_i)_{n}x_n\in R^h_1$, each integer $k_i$ is at least the highest-rank of the common generalized eigenvectors $\cap_{j=1}^n V^j[(\z_i)_j]$ and $N_i \in R^h_{k_i - 1}$.
\end{cor}
\proof Let $\mathfrak{m}_{h,\z_i} \se \k[x_0, \dots, x_n]$ be the maximal homogeneous relevant ideals generated by $(\z_j x_0 - x_j)_{j \in \{1, \dots, n\}}$.
By \cite[Thm. 5.3. D.]{IK} a homogeneous polynomial $F \in R_d^h$ can be written as above if and only if the ideal $I_Z = \mathfrak{m}_{h,\z_1}^{k_1} \cap \cdots \cap \mathfrak{m}_{h,\z_s}^{k_s}$ annihilates $F$ , i.e. $I_Z \se \Ann(F) = \{g \in \k[x_0, \dots, x_n] \, | \, g(\d)(f) = 0 \}$.
In our setting it is sufficient to prove it in the affine chart defined by $x_0 = 1$, in fact we show that
\begin{equation*}
    \mathfrak{m}_{\z_1}^{k_1} \cap \cdots \cap \mathfrak{m}_{\z_s}^{k_s} \se I_{\Lambda} \se \Ann(F)_{x_0 = 1}.
\end{equation*}

As for the first inclusion we have that $I_{\Lambda} = Q_1 \cap \dots \cap Q_s$, then it is sufficient to show that for every $i \in \{1, \dots, s\}$ one has $\mathfrak{m}_{\z_i}^{k_i} \se Q_i$ or, equivalently, that $(Q_i)^{\perp} \se (\mathfrak{m}_{\z_i}^{k_i})^{\perp}$.
We know that $(Q_i)^{\perp} = \langle \{\1_{\z_i} \circ \de^{\a} p_i(\d)\}_{\a \in \N^n} \rangle_{\k}$, so by an easy application of the chain rule, every generator of $(Q_i)^{\perp}$ vanishes on $\m_i^{k_i}$, for all $k_i \geq \deg p_i+1$. 
By Theorem \ref{Thm:CactusStructure} 
the highest-rank $m_i$ of the common generalized eigenvectors $\cap_{j=1}^n V^j[(\z_i)_j]$ is such that $m_i\leq\deg p_i +1$. 
Hence there exists $k_i\geq m_i$ such that $(Q_i)^{\perp} \se (\mathfrak{m}_{\z_i}^{k_i})^{\perp}$.

To prove the second inclusion we notice that, by \cite[Lemma 2.15]{IK}, asking a homogeneous relevant ideal $I$ to be contained in $\Ann(F)$ is equivalent to require $F$ annihilating $I_d$, i.e. the degree $d$ part of $I$. In the affine setting, this translates into $f^* \in \big((I_{\Lambda})_{\leq d}\big)^\perp$, which is true by definition since for every $h \in I_{\Lambda} \cap R_{\leq d}$ we have
\begin{equation*}
    \langle f, h \rangle = f^*(h) = \Lambda(h) = h \star \Lambda(1) = 0.
\end{equation*}

Hence, we conclude that $I_Z \se \Ann(F)$ from which the required decomposition follows.
\endproof

The above discussion leads us to the following algorithm.
\vspace{0.5cm}
\begin{mdframed}[]
\begin{alg}[Cactus rank and decomposition]\label{alg:cactus}\end{alg}
\vspace{0.1cm}
\noindent\textbf{Input:} A degree $d \geq 2$ polynomial $F \in R^h_d$ written by using a general set of essential variables. \\
    \textbf{Output:} The cactus rank of $F$, the $\zeta_1, \ldots ,\zeta_s\in \k^n$ appearing in a generalized decomposition of $\Lambda$ extending $f^*$, their multiplicities and a cactus decomposition of $F$.
\begin{enumerate}
\item Construct the matrix $\H_{\Lambda}(\mathbf{h})$ with the parameters $\mathbf{h} = \{h_{\a}\}_{\substack{\a \in \N^n\\\hspace{-0.1cm}|\a| > d}}$.
\item  Set $r := \rk \rH_{f^*}$.
\item \label{RankLoop} For $B \in \B_{r}$ and $|B|=r$ do
\begin{itemize}
\item Find parameters $\mathbf{h}$ such that:\\
- $\det \H_{\Lambda}^B \neq 0$.\\
- The operators $(\M_i^B)^t := \H^B_{x_i \star \Lambda}(\H^B_{\Lambda})^{-1}$ commute.
\item If one finds such parameters then go to step \ref{CactusLastStep}.
\end{itemize}
\item \label{TestNextr} Set $r := r+1$ and restart step \ref{RankLoop}.
\item \label{CactusLastStep} - Compute the common eigenvectors $v_1, \ldots , v_s$ of the $(\M_j^B)^t$'s, define \\ $\{\z_i = \big(\frac{(v_i)_2}{(v_i)_1}, \dots, \frac{(v_i)_{n+1}}{(v_i)_1}\big)\}_{i \in \{1, \dots, s\}}$ and compute the generalized \\ eigenspaces $\{V^j[(\z_i)_j]\}_{i \in \{1, \dots, s\}}$ of $(\M_{j}^B)^t$ relative to $(\z_i)_j$'s.\\
- Define $r_i = \dim_{\k} \cap_{j = 1}^n V^j [(\z_i)_j]$, let $\overline{k_i}$ be the highest-rank among the common generalized eigenvectors $\cap_{j=1}^nV^j[(\z_i)_j]$ and define $\{L_i = (v_i)_1x_0 + \dots + (v_i)_{n+1}x_n\}_{i \in \{1, \dots, s\}}$.\\
- Find the minimal $k_i \geq \overline{k_i}$ such that the linear system
$$F=\sum_{i=1}^s L_i^{d-k_i+1}\left(\sum_{|\alpha|=k_i-1} \lambda_{\alpha}\mathbf{x}^{\alpha}\right)$$
is solvable in the variables $\lambda_{\alpha} \in \k$.
Then return:
\begin{itemize}
    \item The cactus rank of $F$: $r=\sum_{i=1}^{s} r_i$.
    \item The points on which a generalized decomposition is supported: $\{\z_i\}_{i \in \{1, \dots, s\}}$.
    \item The correspondent multiplicities of these points: $\{r_i\}_{i \in \{1, \dots, s\}}$.
    \item A cactus decomposition of $F$: $\sum_{i=1}^s L_i^{d-k_i+1}\left(\sum_{|\alpha|=k_i-1} \lambda_{\alpha}\mathbf{x}^{\alpha}\right)$.
\end{itemize} 
\end{enumerate}
\end{mdframed}
\vspace{0.5cm}
As for the decomposition algorithm, we start testing $r$ from the rank of the maximal numerical submatrix of $\H_{\Lambda}(\mathbf{h})$ on. That is because the cactus rank never falls behind this value, as shown in \cite[Corollary 3.3]{bbm}.

Notice that if the $k_i$'s appearing in the exponent of the linear forms in the cactus decomposition may be bigger than the highest-rank among the common generalized eigenvectors $\cap_{j=1}^nV^j[(\z_i)_j]$ (e.g. Example \ref{Ex:UpRank2} of next section).
To find minimal $k_i$'s of step \eqref{CactusLastStep} one might gradually increase their values or begin with a high value of $k_i$'s and factorize the form of degree $k_i-1$ multiplying $L^{d-k_i+1}$ afterwards.

Finally, the same idea of Remark \ref{rmk:RandomCombination} may be applied also for finding common generalized eigenvectors, as in the following lemma.

\begin{lemma}\label{lemma:Combination} Let $\{M_i\}_{i \in \{1, \dots, n\}}$ be commuting linear operators on a vector space $V$ over a field $\k$. Let also $v \in V$ be a generalized rank $r_i \geq 1$ eigenvector for every $M_i$, whose correspondent eigenvalue is $\lambda_i \in \k$. Then for every $\{\gamma_1, \dots, \gamma_n\} \se \k$ there is an integer $r \leq r_1 + \dots + r_n - n + 1$ such that $v$ is a generalized rank $r$ eigenvector of $\sum_{i = 1}^n \gamma_i M_i$, relative to the eigenvalue $\sum_{i = 1}^n \gamma_i \lambda_i$.
\end{lemma}
\proof By hypothesis for every $i$ we have
\begin{equation*}
    (M_i - \lambda_i \1)^{r_i} v = 0,
\end{equation*}
and we prove that 
\begin{equation*}
    \left( \sum_{i = 1}^n \gamma_i M_i - \sum_{i = 1}^n \gamma_i \lambda_i \1 \right)^{\sum_{i = 1}^n r_i - n + 1} v = 0.
\end{equation*}
Since the operators $M_i$ commute, also $M_i - \lambda_i \1$ do. Then
\begin{equation*}
    \left( \sum_{i = 1}^n \gamma_i (M_i - \lambda_i \1) \right)^{\sum_{i = 1}^n r_i - n + 1} = \sum_{|\a| = \sum_{i = 1}^n r_i - n + 1} {|\a| \choose \a} \prod_{i = 1}^n \gamma_i^{\a_i}(M_i - \lambda_i \1)^{\a_i}.
\end{equation*}
By Pigeonhole principle we have $\a_i \geq r_i$ for at least one index $i$ in each piece of the above sum, hence it vanishes on $v$.
\endproof 

\subsection{Some examples}\label{Sec:CactusExamples}
Here we perform some examples of the cactus algorithm.

The first shows how the algorithm deals with irreducible non-linear components, providing us with the rank of well-studied forms with ease.

The second highlights that intersecting the generalized eigenspaces is sometimes essential: the information about multiplicities can not always be recovered by knowing only the eigenvalues multiplicities, although these cases almost never occur after a general change of variables.

The last two examples show how the Jordan forms of the multiplication matrices changes according to the cactus structure of the input polynomials.

\begin{exmp}
Let us consider $F = (x^2+y^2+6xz-8z^2)(4x-y-5z) \in \Cx[x,y,z]$, representing a conic with a tangent line. It is well-known that its Waring rank is $5$ and its cactus rank is $3$ \cite{bgi,lt}. We check the cactus rank via our algorithm, which also shows that such a generalized decomposition is unique. 
We dehomogenize $F$ by $x = 1$ and construct $\H_{\Lambda}(\mathbf{h})$. We start from $r = 3$ and observe that the multiplication matrices
\begin{equation*}
(\M_y^B)^t = \left( \begin{array}{cccc}
       0  &     1   &     0\\
   \frac{9}{128} & -\frac{81}{128} &  \frac{17}{128}\\
-\frac{23}{128} &-\frac{177}{128} & -\frac{15}{128}\end{array}\right), \ \
(\M_z^B)^t = \left( \begin{array}{ccc}
      0   &     0    &    1\\
 -\frac{23}{128} & -\frac{177}{128} & -\frac{15}{128}\\
-\frac{183}{128} & -\frac{17}{128} & -\frac{303}{128}\end{array}\right)
\end{equation*}
commute. This is sufficient to conclude that the cactus rank of $F$ is $3$.

Besides, the unique eigenvector common to $(\M_y^B)^t$ and $(\M_z^B)^t$ is $(   4, -1, -5)$ which lies in a generalized eigenspace of dimension $3$ for both the multiplication operators, hence a generalized decomposition of $\Lambda$ is supported on $\z_1 = (-\frac{1}{4}, -\frac{5}{4})$ and its correspondent multiplicity is $r_1 = 3$.
The maximum rank among the common generalized eigenvectors is $k_1 = 3$, hence we conclude $F = L_1^{3-3+1}N_1$, where $N_1 \in \Cx[x,y,z]$ is an homogeneous form of degree $2$ that can be recovered by solving

\begin{small}\begin{equation*}
    F = (4x - y - 5z)(\lambda_{(2,0,0)}x^2 + \lambda_{(0,2,0)}y^2 + \lambda_{(0,0,2)}z^2 + \lambda_{(1,1,0)}xy + \lambda_{(1,0,1)}xz + \lambda_{(0,1,1)}yz).
\end{equation*}
\end{small}

Finally, since no variables $\mathbf{h}$'s had to be chosen, such a cactus decomposition of $F$ is unique and it coincides with the expression of $F$ we started with.
\end{exmp}
\begin{exmp}
Let $F = (x+z)^5x + (x+y-z)^5x + (x+y+z)^6 + (x-z)^6 \in \Cx[x,y,z]$, dehomogenize it by $x = 1$ and construct $\H_{\Lambda}(\mathbf{h})$. The starting $r$ prescribed by the algorithm is $r = 6$, in fact the choice $B = \{1,y,z,y^2,yz,z^2\}$ leads to commuting $(\M_y^B)^t$ and $(\M_z^B)^t$, then the cactus rank of $F$ is 6.
The common eigenvectors are
\begin{align*}
    v_1 &= (1, 0, 1, 0, 0, 1), \quad v_2 = ( 1,  1, -1,  1, -1,  1), \\
    v_3 &= (1, 1, 1, 1, 1, 1), \quad v_4 = ( 1,  0, -1,  0,  0,  1),
\end{align*}
which reveal the points $\z_1 = (0,1),\ \z_2 = (1,-1),\ \z_3 = (1,1),\ \z_4 = (0,-1)$.
The generalized eigenspaces of $(\M_y^B)^t$ are

{\small \begin{equation*}
    V^y[0] = \left\langle
    \left( \begin{array}{c}1 \\ 0 \\ 0 \\ 0 \\ 0 \\ 0 \end{array}\right)
    \left( \begin{array}{c}0 \\ 0 \\ 1 \\ 0 \\ 0 \\ 0\end{array}\right)
    \left( \begin{array}{c}0 \\ 0 \\ 0 \\ 0 \\ 0 \\ 1 \end{array}\right)
    \right\rangle, \quad
    V^y[1] = \left\langle
    \left( \begin{array}{c}1 \\ 1 \\ 0 \\ 1 \\ 0 \\ 1 \end{array}\right)
    \left( \begin{array}{c}0 \\ 0 \\ 1 \\ 0 \\ 1 \\ 0\end{array}\right)
    \left( \begin{array}{c}0 \\ 1 \\ 0 \\ 2 \\ -1 \\ 2 \end{array}\right)
    \right\rangle, 
\end{equation*} }

\noindent whereas the generalized eigenspaces of $(\M_z^B)^t$ are

{ \small \begin{equation*}
    V^z[1] = \left\langle
    \left( \begin{array}{c} 1 \\ 0 \\ 0 \\ 0 \\ 0 \\ -1 \end{array}\right)
    \left( \begin{array}{c}0 \\ 1 \\ 0 \\ 1 \\ 1 \\ 0 \end{array}\right)
    \left( \begin{array}{c}0 \\ 0 \\ 1 \\ 0 \\ 0 \\ 2 \end{array}\right)
    \right\rangle,
    V^z[-1] = \left\langle
    \left( \begin{array}{c}1 \\ 0 \\ 0 \\ -1 \\ 1 \\ -1 \end{array}\right)
    \left( \begin{array}{c}0 \\ 1 \\ 0 \\ 1 \\ -1 \\ 0 \end{array}\right)
    \left( \begin{array}{c}0 \\ 0 \\ 1 \\ -1 \\ 1 \\ -2 \end{array}\right)
    \right\rangle.
\end{equation*} }

Hence, we conclude
\begin{align*}
    r_1 &= \dim_{\k} V^y[0] \cap V^z[1] = 2, \quad r_2 = \dim_{\k} V^y[1] \cap V^z[-1] = 2, \\
    r_3 &= \dim_{\k} V^y[1] \cap V^z[1] = 1, \quad r_4 = \dim_{\k} V^y[0] \cap V^z[-1] = 1.
\end{align*}
We notice that the same result follows even faster by applying Lemma \ref{lemma:Combination}: we consider a random linear combination $\M = \a (\M_y^B)^t + \b (\M_z^B)^t$, then $v_3,v_4$ are simple eigenvectors of $\M$ while $v_1,v_2$ correspond to eigenvalues with algebraic multiplicity $2$, from which the points multiplicities follow immediately.

The above discussion also shows that $k_1 = k_2 = 2$ and $k_3 = k_4 = 1$, then in a cactus decomposition of $F$ the linear forms corresponding to $\z_1$ and $\z_2$ will appear with exponents $6-2+1 = 5$, whereas those coming from $\z_3$ and $\z_4$ will have exponents $6-1+1 = 6$. The missing pieces of the decomposition may be equivalently recovered by solving the final linear system or by looking at generalized eigenvectors, as shown in Section \ref{Section:Tangential}.
\end{exmp}

\begin{exmp}\label{Ex:UpRank3}
Let us apply Algorithm \ref{alg:cactus} on the polynomial
\begin{equation*}
    F = (x+y+z)^2(x^2+y^2+6xz-8z^2) \in \Cx[x,y,z].
\end{equation*}
By choosing $B = \{1,y,z,y^2\}$ there is only one way to fill variables in order to have commuting multiplication operators. The unique common rank-1 eigenvector is (1,1,1,1), which corresponds to the linear form $x+y+z$.

For both the operators the rank-2 eigenspace relative to the eigenvalue $1$ is given by $\langle ( 1,  0,  0, -1), ( 0,  1,  0,  2), ( 0,  0,  1,  0)\rangle$, while the rank-3 eigenspace is the whole space.
Hence, the highest-rank among the generalized eigenvectors is $3$, which in facts leads to the decomposition 
\begin{equation*}
    F = (x+y+z)^{4-3+1}Q,
\end{equation*}
for $Q$ an irreducible quadric that may be computed by solving a linear system.
\end{exmp}

\begin{exmp}\label{Ex:UpRank2}
Here we apply Algorithm \ref{alg:cactus} on the polynomial
\begin{equation*}
    F = (x+y+z)^2(x+y)(x+z) \in \Cx[x,y,z].
\end{equation*}
By choosing $B = \{1,y,z,yz\}$ there is only one way to fill variables in order to have commuting multiplication operators. The unique common rank-1 eigenvector is (1,1,1,1), which corresponds to the linear form $x+y+z$.

For both the operators the rank-2 eigenspace relative to the eigenvalue $1$ is the whole space,
so the highest-rank among the generalized eigenvectors is $\overline{k} = 2$. However, the system 
\begin{equation*}
    F = (x+y+z)^{4-2+1}L
\end{equation*}
for a linear form $L$ is not solvable. Therefore, we consider $k=3$ and see that
\begin{equation*}
    F = (x+y+z)^{4-3+1}Q
\end{equation*}
is solvable for a reducible quadric $Q$.
\end{exmp}

\section{Conclusions and further work}\label{conclusions}
\subsection{How to choose the $\mathbf{h}$'s}
An inheritance that all the presented algorithms takes from \cite{SymTensorDec} is the choice of the variables $\mathbf{h}$'s in the Hankel matrix. This is the crucial and most computationally expensive part of the proposed algorithms, hence the most obvious part to be better understood in order to considerably speed them up. 
These variables are of fundamental importance, since they represent the actual novelty introduced in \cite{SymTensorDec}. In fact, if the Waring decomposition of the given polynomial can be computed without using any $\mathbf{h}$'s, then our algorithm is nothing else than a computational way of realizing the classical Iarrobino-Kanev idea (cf. \cite[5.4]{IK} also rephrased in \cite[Algorithm 2]{oo}).
 
\smallskip

This choice of values for the variables may occur only in Step \ref{MainLoop} of all the algorithms.
Firstly, when we search for parameters $\mathbf{h}$'s such that $\det \H_{\Lambda}^B \neq 0$. Asking this determinant to be non-zero is an open condition (in the Zariski topology), call it $C_1^0$.
The second time in which another choice of $\mathbf{h}$'s has to be done is when we ask the operators $(\M_i^B)^t = \H^B_{x_i \star \Lambda}(\H^B_{\Lambda})^{-1}$ to commute. This is a closed condition, call it $C_2$, which assures that $ \H_{x_i \star \Lambda}^B (\H_{\Lambda}^B)^{-1}
$ is actually the matrix of the multiplication-by-$x_i$ operator $M_{x_i}^t$ as seen in Lemma  \ref{MultiplicationHankel}.
We conjecture that $C_1^0$ and $C_2$ are somehow independent, meaning that first one does not affect the (im)possibility of the second, but this would certainly deserve further investigation.

Consider all the $\mathbf{h}$'s that one can find asking that both $C_1^0$ and $ C_2$ are satisfied. Each of them leads to a different $\Lambda$ which extends $f^*$. In the case of the cactus algorithm we know a priori that once conditions $C_1^0$ and $ C_2$ are both satisfied then we will get cactus rank.
This implies that if we are interested in computing either the cactus rank or a cactus decomposition any choices of $\mathbf{h}$'s satisfying $C_1^0 \cap C_2$ will make the algorithm terminate.

\smallskip

Instead, if we are looking either for a Waring decomposition or for a tangential one, $C_1^0 \cap C_2$ is not sufficient to know in advance which of these $\Lambda$'s will give rise to the correct number of eigenvectors common to all $(\mathbb{M}_i)^t$'s. 
The closed condition $C_2$ defines an ideal $I_{C_2}$ which itself defines a variety $\VV(I_{C_2})$, which may have many irreducible components: ${\VV(I_{C_2})}= V_1 \cup \cdots \cup V_k$.
The points in every $V_i$ correspond to $\Lambda$'s extending $f^*$. Each of these $\Lambda$'s gives rise to an ideal $I_{\Lambda}=\ker H_\Lambda$ (Definition \ref{def:Hankel}) defining a scheme of length $r=\dim(\A_{\Lambda})$, which is associated to a non empty $\VV(I_\Lambda)$ as we have seen in Section \ref{section:cactus:rank}.

Thus, a random choice of $\mathbf{h}$'s corresponds to pick  a random point in a component $V_i$ with the highest dimension. For this reason, 
in the unlucky case that a random choice of the $\mathbf{h}$'s does not lead to a required decomposition, then it is not sufficient to check other randomly chosen values of $\mathbf{h}$'s but a different component of $\VV(I_{C_2})$ has to be considered. 
After all these components had been tested we can conclude that $r$ has to be increased in order to find a required decomposition 
since in \cite[Theorem 3.16]{abm} it is proved that $C_2$ are quadratic equations for the punctual Hilbert scheme of points in $\mathbb{P}^n$ of length $r$.

This procedure works in general since the smooth elements in the Hilbert scheme of points of length $r$ are generic in their connected component. Therefore this procedure  provides an effective way to find a good choice of the $\mathbf{h}$'s and to decide that this cannot be achieved with the considered rank.

This is an actual proof-of-work of the proposed algorithms.

\subsection{Choice of the basis} 
In Section \ref{requirementB} we discussed the improvements obtained by choosing complete staircase bases $B$ for $\A_{\Lambda}$ instead of bases which are only connected to 1.
However, one can conceivably think to stricter criteria on these bases in order to reduce the number of tests performed by the algorithms even further. 
As an instance, Borel-fixed monomial bases \cite[Chapter 2]{BorelFixed} appear to be interesting candidates for this scope. 
A concrete example of such an improvement regards the high-degree monomials: we proved in Theorem \ref{CompleteStaircase} that for finding a Waring decomposition of a degree $d$ polynomial, the bases $B$ can be chosen involving only monomials of degree at most $d$, but we suspect that degree $d-1$ suffices. From a computational point of view, this would drastically reduce the number of $\mathbf{h}$'s to be chosen by decreasing the average degree of the elements in $B$.
Besides, this would also have theoretical implications, leading to another short proof of the well-known bound on the maximum rank of generic polynomials, i.e. ${n+d\choose d}-n$ (cf. \cite{ger,lt}). One may hope to improve this bound for certain families of polynomials by refining the algebraic constraints on their bases even more.

Another challenging open problem about these bases regards the degrees of their elements for both tangential and cactus decomposition, as we outlined in Remark \ref{Rmk:degreeofB}. We have proved that such bases may always be made of elements of degree up to $\deg F$ only in the Waring case, but we lack a proof or a counterexample of requiring a similar condition even in the tangential setting, even more so in the cactus one.

\subsection{Cactus decomposition}
In Section \ref{section:cactus:rank} we presented an algorithm that computes the cactus rank and returns each $\1_{\z_i}$ together with the multiplicities of the $\z_i$'s. 
This algorithm terminates by solving the linear system
$$ F=\sum_{i=1}^s L_i^{d-{k_i}+1}\left(\sum_{|\alpha|=k_i-1} \lambda_{\alpha}\mathbf{x}^{\alpha}\right),$$
in the (possibly many) unknowns $\lambda_{\alpha}$'s.
It would be compelling to recover a priori more information on the $k_i$'s and on the actual polynomials $\sum_{|\alpha|=k_i-1} \lambda_{\alpha}\mathbf{x}^{\alpha}$ by recovering the $p_i$'s appearing in the generalized decomposition of $f^*$: it would decrease the number of parameters involved in the above linear system. From \cite{bjmr} the forms multiplying each $L_i^{d-k_i+1}$ define the local Gorenstein scheme ``ackwardly'' defined via dehomogenization by $L_i$. We strongly believe that this fact will lead to a deeper understanding on the reconstruction of the cactus decomposition.

Besides, as suggested by Examples \ref{Ex:UpRank3} and \ref{Ex:UpRank2}, the Jordan forms of the multiplication operators seem to change according to the algebraic properties (e.g. reducibility) of these forms. It would be awesome to exploit this knowledge in order to recover more information about the cactus structure of the considered polynomial in advance.

\section*{Acknowledgements}

We warmly thank A. Iarrobino and B. Mourrain for many useful conversations and suggestions.  We also thank the anonymous reviewer for his/her careful reading and insightful comments.

AB acknowledges financial support from GNSAGA of INDAM.

\end{document}